\documentclass[a4paper,12pt]{amsart}
\pdfoutput=1

\usepackage{mathtools}
\mathtoolsset{showonlyrefs=true}

\usepackage{amsmath,amssymb,amsthm,amsfonts, mathrsfs, fancyhdr}
\usepackage[usenames,dvipsnames,usenames,dvipsnames,svgnames,table]{xcolor}
\usepackage[expansion=false]{microtype}
\usepackage[all]{xy}

\usepackage{upgreek}

\usepackage{enumitem}
\makeatletter
\newcommand{\mylabel}[2]{#2\def\@currentlabel{#2}\label{#1}}
\makeatother

\setlist[description]{leftmargin=*}
\usepackage{amscd}
\usepackage[active]{srcltx}
\usepackage{verbatim}
\usepackage{epsfig,graphicx,color}
\usepackage{graphicx}
\usepackage{mathtools,xparse}
\usepackage[english]{babel}

\usepackage{pgf, tikz, tikz-cd}
\usepackage{dsfont}
\usetikzlibrary{matrix}

\usepackage{braids}

\usepackage{capt-of}
\usepackage{float}

\usepackage{verbatim}
\usepackage{cite}
\usepackage{manyfoot}

\usepackage[left=2.7cm,right=2.7cm,top=3cm,bottom=3cm]{geometry}

\usepackage[unicode,bookmarks, pdftex, backref = page]{hyperref}
\definecolor{newblue}{RGB}{0,102,204}
\hypersetup{colorlinks=true,citecolor=newblue,linkcolor=newblue, urlcolor=Orange,
pdfpagemode=UseNone, breaklinks=true}

\newtheorem{theorem}{Theorem}[section]
\newtheorem{proposition}[theorem]{Proposition}
\newtheorem{lemma}[theorem]{Lemma}
\newtheorem{corollary}[theorem]{Corollary}
\newtheorem{question}[theorem]{Question}

\newtheorem*{A}{Theorem A}
\newtheorem*{B}{Theorem B}
\newtheorem*{Ci}{Theorem C}
\newtheorem*{D}{Theorem D}

\theoremstyle{definition}
\newtheorem{remark}[theorem]{Remark}

\newtheorem{definition}[theorem]{Definition}

\newcommand{\ZZ}{\mathbb{Z}}

\newcommand{\z}{\mathsf{z}}
\renewcommand{\to}{\longrightarrow}

\usepackage{epigraph}
\title[Finite Heisenberg groups and double Kodaira fibrations]{Surface braid groups, finite Heisenberg covers and double Kodaira fibrations}

\date{}

\author[Andrea Causin]{Andrea Causin}
\address{Dipartimento di Architettura, Design e Urbanistica
\newline\indent
Universit\`{a} degli Studi di Sassari
\newline\indent
Palazzo del Pou Salit, Piazza Duomo 6, 07041 Alghero, Sassari, Italy.}
\email{acausin@uniss.it}

\author[Francesco Polizzi]{Francesco Polizzi}
\address{Dipartimento di Matematica e Informatica
\newline\indent
Universit\`a della Calabria
\newline\indent
Ponte Pietro Bucci 30B, I-87036 Arcavacata di Rende, Cosenza, Italy}
\email{polizzi@mat.unical.it}



\begin{document}



\begin{abstract}
We exhibit new examples of double Kodaira fibrations by using finite Galois covers of a product $\Sigma_b \times \Sigma_b$, where $\Sigma_b$ is a smooth projective curve of genus $b \geq 2$.  Each cover is obtained by providing an explicit group epimorphism from the pure braid group $\mathsf{P}_2(\Sigma_b)$ to some finite Heisenberg group. In this way, we are able to show that every curve of genus $b$ is the base of a double Kodaira fibration; moreover, the number of pairwise non-isomorphic Kodaira fibred surfaces fibering over a fixed curve $\Sigma_b$ is at least $\boldsymbol{\upomega}(b+1)$, where $\boldsymbol{\upomega} \colon \mathbb{N} \to \mathbb{N}$ stands for the arithmetic function counting the number of distinct prime factors of a positive integer. As a particular case of our general construction, we obtain a real $4$-manifold of signature $144$ that can be realized as a real surface bundle over a surface of genus $2$, with fibre genus $325$, in two different ways.
This provides (to our knowledge) the first ``double solution" to a problem from Kirby's problem list in low-dimensional topology.

\end{abstract}

\maketitle

\tableofcontents





\setcounter{section}{-1}

\section{Introduction} \label{sec:intro}

A \emph{Kodaira fibration} is a smooth, connected holomorphic fibration $f_1 \colon S \to B_1$, where $S$ is a compact complex surface and $B_1$ is a compact complex curve, which is not isotrivial (this means that not all its fibres are biholomorphic to each others). Equivalently, by \cite{FiGr65}, a Kodaira fibration is a smooth connected fibration $f_1 \colon S \to B_1$ which is not a locally trivial holomorphic fibre bundle, cf.\cite[Chapter I, (10.1)]{BHPV03}; however, any such a fibration is a locally trivial differentiable fibre bundle in the category of real $C^{\infty}$ manifolds, because of the Ehresmann theorem \cite{Eh51}. The genus $b_1:=g(B_1)$ is called the \emph{base genus} of the fibration, whereas the genus $g:=g(F)$, where $F$ is any fibre, is called the \emph{fibre genus}. If a surface $S$ is the total space of a Kodaira fibration, we will call it a \emph{Kodaira fibred surface}.

By \cite[Theorem 1.1]{Kas68}, every Kodaira fibration $f_1 \colon S \to B_1$ satisfies $b_1 \geq 2$ and $g \geq 3$. In particular, $S$ contains no rational or elliptic curves: in fact, such curves can neither dominate the base (because $b_1 \geq 2$) nor be contained in fibres (since the fibration is smooth). So every Kodaira fibred surface $S$ is minimal and, by the superadditivity of the Kodaira dimension, it is of general type, hence algebraic.

Examples of Kodaira fibrations were originally constructed in \cite{Kod67} in order to show that, unlike the topological Euler characteristic, the signature $\sigma$ of a manifold, meaning the signature of the intersection form in the middle cohomology, is not multiplicative for fibre bundles. Indeed, every Kodaira fibred surface $S$ satisfies $\sigma(S) >0$, see for example the introduction of \cite{LLR17}, whereas $\sigma(B_1)=\sigma(F)=0$, and so $\sigma(S) \neq \sigma(B_1)\sigma(F)$. On the other hand, in \cite{CHS57} it is proved that the signature is multiplicative for fibre bundles in the case where the monodromy action of the fundamental group $\pi_1(B_1)$ on the rational cohomology ring $H^*(F, \, \mathbb{Q})$ is trivial; thus, Kodaira fibrations provide examples of fibre bundles for which this action is non-trivial. In fact, the non-triviality of the monodromy action is ensured by the non-isotriviality of the fibration, see also the variants of Kodaira's construction later presented by Atiyah \cite{At69} and Hirzebruch \cite{Hir69}. In this regard, Kodaira fibrations show the important role that the fundamental group plays at the cross-road between the algebro-geometric properties of a complex surface and the topological properties of the underlying real $4$-manifold.

The technique used by Kodaira, Atiyah and Hirzebruch, namely taking a suitable ramified cover of a product of curves, was subsequently refined in a series of papers by several authors, see  \cite{Zaal95,LeBrun00,BDS01,BD02,CatRol09,Rol10,LLR17}. Moreover, very recently, the study of the monodromy action from a Hodge-theoretical point of view has been undertaken in some particular cases, see \cite{Fl17, Breg18}.

There are also different constructions, whose flavour is more complex-analytic, based on Teichm{\"u}ller theory and Bers fibre spaces \cite{GDH91a} or theta functions and Siegel modular forms \cite{GDH91b}. Furthermore, topological characterizations of Kodaira fibrations in terms of Euler characteristic and fundamental group can be found in \cite{Hil00, CatRol09}.

In the context of real $4$-manifolds, the related problem of constructing non-trivial, real surface bundles over surfaces has been addressed by using techniques from geometric topology such as the Meyer signature formula, the Birman-Hilden relations in the mapping class group and the subtraction of Lefschetz fibrations, see for example \cite{En98, EKKOS02, St02, L17}.

For more details on these (and many other) topics, we refer the reader to the survey paper \cite{Cat17} and to the references contained therein.

It is rather difficult to construct  Kodaira fibrations with small $\sigma(S)$. Since $S$ is a differentiable $4$-manifold which is a real surface bundle, its signature is divisible by $4$, see \cite{Mey73}. If moreover $S$ has a spin structure, i.e., its canonical class is $2$-divisible in $\mathrm{Pic}(S)$, then by Rokhlin's theorem its signature is necessarily a positive multiple of $16$, and examples with $\sigma(S)=16$ are constructed in \cite{LLR17}. It is not known if there exists a  Kodaira fibred surface such that $\sigma(S) \leq 12$.

Another important invariant of Kodaira fibred surfaces is the slope $\nu(S)=c_1^2(S)/c_2(S)$, that can be seen as a quantitative measure of the non-multiplicativity of the signature. In fact, every product surface $F \times B_1$ satisfies $\nu(F \times B_1)=2$; on the other hand, if $S$ is a Kodaira fibred surface, then Arakelov inequality (see \cite{Be82}) implies $\nu(S)>2$, while Liu inequality (see \cite{Liu96}) yields $\nu(S)<3$, so that for such a surface the slope lies in the open interval $(2, \, 3)$. The original examples by Atiyah, Hirzebruch and Kodaira have slope lying in $(2, 2 + 1/3]$, see \cite[p. 221]{BHPV03}, and the first examples with higher slope appeared in \cite{CatRol09}, where it is shown that there are Kodaira surfaces satisfying $\nu(S)=2+2/3$. This is the record for the slope so far, in particular it is a present unknown whether the slope of a Kodaira fibred surface can be arbitrarily close to $3$.

The examples provided in \cite{CatRol09} are actually rather special cases of Kodaira fibred surfaces, called \emph{double Kodaira surfaces}. 

\begin{definition}
A \emph{double Kodaira surface} is a compact complex surface $S$, endowed with a \emph{double Kodaira fibration}, namely a surjective, holomorphic map $f \colon S \to B_1 \times B_2$ yielding, by composition with the natural projections, two Kodaira fibrations $f_i \colon S \to B_i$, $i=1, \,2$.
\end{definition}
Note that a surface $S$ is a double Kodaira surface if and only if it admits two distinct Kodaira fibrations $f_i \colon S \to B_i$, $i=1, \,2$, since in this case we can take as $f$ the product morphism $f_1 \times f_2$.

\medskip
Let us now describe our approach to these topics, which exploits the techniques introduced in \cite{Pol18}, and present our results. If $\Sigma_b$ denotes a smooth projective curve of genus $b$, the aim of the present paper is to construct new examples of double Kodaira fibrations by taking some \emph{Heisenberg covers} of the product $\Sigma_b \times \Sigma_b$, namely some Galois covers, branched over the diagonal $\Delta \subset \Sigma_b \times \Sigma_b$ and whose Galois group is isomorphic to some finite Heisenberg group. More precisely, we take any odd prime number $p$ and we consider the isomorphism of  $\mathbb{Z}_p$-vector spaces
\begin{equation}
V = H_1(\Sigma_b \times \Sigma_b - \Delta, \mathbb{Z}_p) \simeq H_1(\Sigma_b \times \Sigma_b, \, \mathbb{Z}_p) \simeq (\mathbb{Z}_p)^{4b},
\end{equation}
see Subsection \ref{subsec:config-integer}. The space $V$ is endowed with an alternating form $\omega \colon V \times V \to \mathbb{Z}_p$ such that $\ker \omega = V_0 \subset V$; correspondingly, there is a central extension
\begin{equation}
\label{eq-intro:central}
1 \to \mathbb{Z}_p \to \mathsf{Heis}(V, \, \omega) \to V \to 1,
\end{equation}
see Definition \ref{def:Heis} and Remark \ref{rmk: degenerate 2-form}. Fixing any ordered set $\mathscr{P}=\{p_1, \, p_2\}$ of two points in $\Sigma_b$, the fundamental group $\pi_1(\Sigma_b \times \Sigma_b - \Delta, \, \mathscr{P})$ is isomorphic to the group $\mathsf{P}_2(\Sigma_b)$ of pure braids on two strings based at $\mathscr{P}$, see Definition \ref{def:braid}; furthermore, there is a natural surjective group homomorphism 
\begin{equation}
\phi \colon \mathsf{P}_2(\Sigma_b) \longrightarrow  V,
\end{equation}
given by the composition of the reduction mod $p$ map $H_1(\Sigma_b \times \Sigma_b - \Delta, \, \mathbb{Z}) \to V$ with the abelianization map $\mathsf{P}_2(\Sigma_b) \longrightarrow  H_1(\Sigma_b \times \Sigma_b - \Delta, \, \mathbb{Z})$. Let $A_{12}$ be  the generator of $ \mathsf{P}_2(\Sigma_b)$ given by the homotopy class of a loop in $\Sigma_b \times \Sigma_b - \Delta$, based at $\mathscr{P}$ and  that ``winds once around $\Delta$''; then,  in Theorem \ref{thm:interpretation-obstruction}, we provide necessary and sufficient conditions on $\omega$ ensuring that $\phi$ admits a surjective lifting
\begin{equation}
\varphi_{\omega} \colon \mathsf{P}_2(\Sigma_b) \to \mathsf{Heis}(V, \omega)
\end{equation}
such that $\varphi_{\omega}(A_{12})$ is non-trivial. This in turn allows us (by using the Riemann Extension Theorem, see Proposition \ref{prop:Riemann-ext}) to construct a finite Heisenberg cover $\mathbf{f} \colon S_{b, \, p} \to \Sigma_b \times \Sigma_b$, depending on $\omega$, whose description is achieved by considering separately two cases: the one where $\omega$ is non-degenerate (i.e., symplectic) and the one where it is degenerate.

When $\omega$ is symplectic, i.e. $V_0=(0)$, the fibres of the compositions $S_{b, \, p} \stackrel{\mathbf{f}}{\to} \Sigma_b \times \Sigma_b \stackrel{\pi_i}{\to} \Sigma_b$ are not connected, so we need to perform a Stein factorization in order to obtain a double Kodaira fibration $f \colon S_{b, \, p} \to \Sigma_{b'} \times \Sigma_{b'}$, which turns out to be a $\mathbb{Z}_p$-cover which is ``very simple'' in the sense of Definition \ref{def:simple-very-symple-standard}. This is the content of Theorem \ref{thm:double-Kodaira-from-omega}, Proposition \ref{prop:degree-double-Kodaira} and Proposition \ref{prop:invariants-nondegenerate-case}, that for the sake of brevity we present here in the following condensed version.

\begin{A}
For every positive integer $b \geq 2$ and every prime number $p \geq 5$, there exists a double Kodaira fibration  $f \colon S_{b, \, p} \to \Sigma_{b'} \times \Sigma_{b'}$, where
\begin{equation}
b'-1 = p^{2b}(b-1),
\end{equation}
which is a cyclic cover of degree $p$, branched over the disjoint union of $p^{2b}$ graphs of automorphisms. The two Kodaira fibrations $f_i \colon S_{b, \, p} \to \Sigma_{b'}$ have the same fibre genus $g$, which is related to $b$ by the formula
\begin{equation} 
2g-2 = p^{2b+1} \left(2b-2 + \mathfrak{p}\right),
\end{equation}
where $\mathfrak{p}:= 1 - 1/p$. The slope of $S_{b, \, p}$ is given by
\begin{equation*}
\nu(S_{b, \, p}) = 2+ \frac{2 \mathfrak{p} - \mathfrak{p}^2}{2b-2 + \mathfrak{p} },
\end{equation*}
and so we have
\begin{equation}
2  < \nu(S_{b, \, p} ) \leq 2 + \frac{12}{35},
\end{equation}
with equality on the right holding precisely when $(b, \,p) \in \{ \,(2, \,5), \; (2,\, 7) \, \}$. 

\medskip

It follows $\nu(S_{2, \, p}) > 2 + 1/3$ for all $p \geq 5$. More precisely, if $p \geq 7$ the function $\nu(S_{2, \, p})$ is strictly decreasing and
\begin{equation*}
\lim_{p \rightarrow +\infty} \nu(S_{2, \, p}) = 2 + \frac{1}{3}.
\end{equation*}
Finally, for all pairs $(b, \,p)$ we have
\begin{equation}
\sigma(S_{b, \, p}) = \frac{1}{3} p^{4b+1}(2b-2)(2 \mathfrak{p} - \mathfrak{p}^2) \geq \sigma(S_{2, \, 5}) = 2^4 \cdot 5^7
\end{equation}
\end{A}

Regarding the case where $\omega$ is degenerate, we only analyze the case where the matrix representing $\omega$ has the form
\begin{equation} \label{eq:intro-omega-degenerate}
\begin{pmatrix}
J_b & J_b \\
J_b & J_b
\end{pmatrix} \in \mathrm{Mat}_{4b}(\mathbb{Z}_p),
\end{equation}
where $J_b$ is the standard $2b \times 2b$ symplectic matrix. Then $V_0= \ker \omega$ has  dimension $2b$ and, if $p$ divides $b+1$, there is a group epimorphism
\begin{equation}
\varphi_{\omega} \colon \mathsf{P}_2(\Sigma_b) \to \mathsf{Heis}(W, \omega),
\end{equation}
where $W=V/V_0$. The corresponding cover, that we call a \emph{degenerate Heisenberg cover}, directly yields a double Kodaira fibration $f \colon S^{\circ}_{b, \, p} \to \Sigma_b \times \Sigma_b$, that is, in this situation no Stein factorization is needed (the ``$\circ$'' superscript is used here in order to emphasize this fact). An analogous result holds when $p=2$, taking the Heisenberg matrix group $\mathsf{H}_{2b+1}(\mathbb{Z}_2)$ described in \eqref{eq:heis-matrices} as a substitute of $\mathsf{Heis}(W, \, \omega)$, whose definition makes no sense in characteristic $2$. Since for $p \geq 3$ the group $\mathsf{H}_{2b+1}(\mathbb{Z}_p)$ is isomorphic to $\mathsf{Heis}(W, \, \omega)$, we can summarize our results by treating both the cases $p$ odd and $p=2$ at once, see Theorem \ref{thm:directly-Kodaira} and  Proposition \ref{prop:invariants-degenerate-case}.

\begin{B}
Let $\Sigma_b$ be any smooth curve of genus $b \geq 2$. Then, for all primes $p$ dividing $b+1$, there exists a double Kodaira fibration $f \colon S^{\circ}_{b, \, p} \to \Sigma_b \times \Sigma_b$. The morphism $f$ is a finite Galois cover of degree $p^{2b+1}$, branched on the diagonal $\Delta \subset \Sigma_b \times \Sigma_b$ and having Galois group isomorphic to the matrix Heisenberg group $\mathsf{H}_{2b+1}(\mathbb{Z}_p)$. Both Kodaira fibrations  $S^{\circ}_{b, \, p} \to \Sigma_b$ have the same fibre genus $g$, satisfying the relation
\begin{equation}
2g-2 = p^{2b+1}(2b-2+ \mathfrak{p}),
\end{equation}
where $\mathfrak{p}:=1-1/p$. Finally, the invariants of $S^{\circ}_{b, \, p}$ are
\begin{equation}
\begin{split}
c_1^2(S^{\circ}_{b, \, p}) & =  p^{2b+1}(2b-2) (4b-4 + 4 \mathfrak{p} - \mathfrak{p}^2) \\
c_2(S^{\circ}_{b, \, p}) & = p^{2b+1}(2b-2)(2b-2 + \mathfrak{p}),
\end{split}
\end{equation}
so that the slope and the signature of $S^{\circ}_{b, \, p}$ can be expressed as
\begin{equation}
\begin{split}
\nu(S^{\circ}_{b, \, p}) & = \frac{c_1^2(S^{\circ}_{b, \, p})}{c_2(S^{\circ}_{b, \, p})} = 2+ \frac{2 \mathfrak{p} - \mathfrak{p}^2}{2b-2 + \mathfrak{p} } \\
\sigma(S^{\circ}_{b, \, p}) & = \frac{1}{3}\left(c_1^2(S^{\circ}_{b, \, p}) - 2 c_2(S^{\circ}_{b, \, p}) \right) = \frac{1}{3} p^{2b+1}(2b-2)(2 \mathfrak{p} - \mathfrak{p}^2).
\end{split}
\end{equation}
In particular, we have
\begin{equation}
\nu(S^{\circ}_{b, \, p}) \leq \nu(S^{\circ}_{2, \, 3})=2 + \frac{1}{3}, \quad \sigma(S^{\circ}_{b, \, p}) \geq \sigma(S^{\circ}_{3, \, 2}) = 128.
\end{equation}
\end{B}

We emphasize that the degenerate construction allows us to obtain a much smaller minimal signature than the construction with the symplectic form, namely $128$ instead of $2^4 \cdot 5^7$. This is due to the fact that in the degenerate situation we can decrease the lowest admissible order of the Heisenberg group from $5^9$ to $2^7$.
The price to pay is that now not all the pairs $(b, \, p)$ are suitable for the construction, but only those with $p$ dividing $b+1$, and this prevents us from obtaining slope higher that $2+ 1/3$ in this way.

We also remark that we can obtain only finitely many values of the fibre genus $g$ for a fixed $b$, because there are only finitely many primes $p$ dividing $b+1$. In fact, all these primes give pairwise distinct covers, as stated by our next result, see Corollary \ref{cor:number-kodaira-any-curve}.

\begin{Ci} \label{thm:ci}
Let $\Sigma_b$ be any smooth curve of genus $b$. Then there exist at least one and at most finitely many  double Kodaira surfaces arising as degenerate Heisenberg covers of the form $f \colon S^{\circ}_{b, \, p} \to \Sigma_b \times \Sigma_b$. Furthermore, such surfaces are pairwise non-homeomorphic and, denoting their number by $\kappa(b)$, we have
\begin{equation*}
\kappa(b) = \boldsymbol{\upomega}(b+1),
\end{equation*}
where $\boldsymbol{\upomega} \colon \mathbb{N} \to \mathbb{N}$ stands for the arithmetic function counting the number of distinct prime factors of a positive integer. In particular, we obtain
\begin{equation*}
\limsup_{b \rightarrow + \infty} \kappa(b) = + \infty.
\end{equation*}
 \end{Ci}

We believe that the results described above are significant for at least three reasons:
\begin{itemize}
\item[$\boldsymbol{(i)}$] the idea of using finite quotients of $\mathsf{P}_2(\Sigma_b)$ in order to obtain double Kodaira fibrations seems to be a new one. Furthermore, both the non-degenerate and the degenerate construction can be applied uniformly for all values of the base genus $b$;
\item[$\boldsymbol{(ii)}$] the non-degenerate construction allows us to obtain in one shot infinitely many double Kodaira fibrations with slope higher than $2 + 1/3$ (in fact, the maximum slope obtained with our method is $2+12/35$), maintaining at the same time a complete control on both the base genus and on the signature (cf. Theorem \ref{thm:double-Kodaira-from-G} for a result in greater generality). By contrast, the (beautiful) ``tautological construction" used in \cite{CatRol09} yields the higher slope $2+ 2/3$, but it involves an {\'e}tale pullback ``of sufficiently large degree'', that completely loses control on the other quantities. On the other hand, the revisitation of the tautological construction presented in \cite{LLR17} makes the pull-back explicit, but it requires a careful case-by-case analysis of the monodromy action in order to make sure that some ``virtual'' topological data actually give rise to an effective double Kodaira fibration;  
\item[$\boldsymbol{(iii)}$] the degenerate construction is a very minimal one, because it provides double Kodaira surfaces  directly as (non-abelian) Galois covers of $\Sigma_b \times \Sigma_b$ branched precisely over the diagonal $\Delta$. As far as we know, this is the first construction showing that \emph{all} curves $\Sigma_b$ of arbitrary genus $b \geq 2$ (and not only some curves with non-trivial automorphisms) are the base of at least one double Kodaira fibration $S \to \Sigma_b \times \Sigma_b$; in addition, the number of topological types of $S$, for a fixed $\Sigma_b$, can be arbitrarily large  (cf. Theorem C). 
\end{itemize}

A particularly interesting instance of $\boldsymbol{(iii)}$ is the case $b=2$, $p=3$. It provides (to our knowledge) the first ``double solution'' to a problem, posed by G. Mess, from Kirby's problem list in low-dimensional topology (\cite[Problem 2.18 A]{Kir97}, see also Remark \ref{rmk:BD02}), asking what is the smallest number $b$ for which there exists a real surface bundle over a surface with base genus $b$ and non-zero signature. This is the content of our next result, see Proposition \ref{prop:Kirby}.

\begin{D}
There exists an oriented $4$-manifold $X$ $($namely, the real $4$-manifold underlying the complex surface $S^{\circ}_{2, \, 3})$ of signature $144$ that can be realized as a real surface bundle over a surface of genus $2$, with fibre genus $325$, in two different ways.
\end{D}
In fact, we may ask whether $144$ and $325$ are the minimum possible values for the signature and the fibre genus of a double Kodaira surface $S \to \Sigma_2 \times \Sigma_2$, see Question \ref{q:minimal-values}.

This work is organized as follows. In Section \ref{sec:configuration} we set up notation and terminology and we collect the background material which is needed in the sequel of the paper. In particular, in Corollary \ref{cor-cohomology-K-conf} we compute the (co)homology groups of the configuration space $\Sigma_b \times \Sigma_b-\Delta$ with coefficient in an arbitrary field $\mathbb{K}$ and in Proposition \ref{prop:cup-product} we show the surjectivity of the corresponding cup-product pairing (all of this is certainly known to the experts, but difficult to find in a self-contained form), whereas in Theorem \ref{thm:presentation-braid} we discuss Gon\c{c}alves-Guaschi's presentation of the braid group $\mathsf{P}_2(\Sigma_b)$, see \cite{GG04}. The main topic of Section \ref{sec:rep-braid} is our lifting result Theorem \ref{thm:interpretation-obstruction}, whose proof is based on group cohomological arguments and that seems to be of independent interest. Finally, all these pieces of information are applied in Section \ref{sec:Kodaira} in order to construct our Heisenberg covers of $\Sigma_b \times \Sigma_b$ and the corresponding double Kodaira fibrations.


\section{Topology of configuration spaces and surface braid groups} \label{sec:configuration}

\subsection{Cohomology of $\Sigma_b \times \Sigma_b - \Delta$: rational coefficients} \label{subsec:config-rational}

Let $\Sigma_b$ be a smooth, projective curve of genus $b \geq 2$ and choose a basis
\begin{equation} \label{eq:symplectic}
\alpha_1, \ldots, \alpha_b, \, \beta_1, \ldots, \beta_b
\end{equation}
for the first cohomology group $H^1(\Sigma_b, \, \mathbb{Z})$, symplectic with respect to the cup product
\begin{equation}
H^1(\Sigma_b, \, \mathbb{Z}) \times H^1(\Sigma_b, \, \mathbb{Z}) \to H^2(\Sigma_b, \, \mathbb{Z}).
\end{equation}
Denoting by $1$ and $\gamma$ the generators of
 $H^0(\Sigma_b, \, \mathbb{Z})$ and $H^2(\Sigma_b, \, \mathbb{Z})$, by using the K{\"u}nneth Theorem and the Universal Coefficient Theorem for cohomology (\cite[Theorem 5.5.8]{We14}) we see that the cohomology groups $H^i(\Sigma_b \times \Sigma_b, \, \mathbb{K})$ over an arbitrary field $\mathbb{K}$ are given by
\begin{equation} \label{eq:cohomology-product}
\begin{split}
H^0(\Sigma_b \times \Sigma_b, \, \mathbb{K})& = \langle 1 \otimes 1 \rangle  \simeq \mathbb{K}\\
H^1(\Sigma_b \times \Sigma_b, \, \mathbb{K})& = \langle 1\otimes \alpha_i, \, 1\otimes \beta_i, \, \alpha_i \otimes 1, \, \beta_i \otimes 1 \rangle \simeq \mathbb{K}^{4b}  \\
H^2(\Sigma_b \times \Sigma_b, \, \mathbb{K}) & = \langle 1 \otimes \gamma, \, \alpha_i \otimes \alpha_j, \, \alpha_i \otimes \beta_j, \, \beta_i \otimes \alpha_j, \, \beta_i \otimes \beta_j, \, \gamma \otimes 1 \rangle \simeq \mathbb{K}^{4b^2+2} \\
H^3(\Sigma_b \times \Sigma_b, \, \mathbb{K}) & = \langle \gamma \otimes \alpha_i, \, \gamma \otimes \beta_i, \, \alpha_i \otimes \gamma, \, \beta_i \otimes \gamma \rangle \simeq \mathbb{K}^{4b} \\
H^4(\Sigma_b \times \Sigma_b, \, \mathbb{K}) & = \langle \gamma \otimes \gamma \rangle \simeq \mathbb{K},
\end{split}
\end{equation}
where $i, \, j \in \{1, \ldots, b\}$ and $\langle \;\; \rangle$ means ``span over $\mathbb{K}$''. 

\begin{proposition} \label{prop:cup-product-easy}
For every field $\mathbb{K}$, the cup-product pairing
\begin{equation} \label{eq:cup-product-easy}
\xi \colon \wedge^2 H^1(\Sigma_b \times \Sigma_b, \, \mathbb{K}) \longrightarrow H^2(\Sigma_b \times \Sigma_b, \, \mathbb{K})
\end{equation}
is surjective.
\end{proposition}
\begin{proof}

The cup product in the graded-commutative tensor algebra  $H^{\ast}(\Sigma_b \times \Sigma_b, \, \mathbb{K}) = H^{\ast}(\Sigma_b, \, \mathbb{K}) \otimes H^{\ast}(\Sigma_b, \, \mathbb{K}) $ is given by
\begin{equation} \label{eq:cup-in-tensor}
\xi(x \otimes y, \, z \otimes w) = (-1)^{\mathrm{deg}(y) \, \mathrm{deg}(z)} xz \otimes yw,
\end{equation}
see \cite[p. 219]{Hat02}. The basis in \eqref{eq:symplectic} being symplectic, we have $\alpha_i \beta_j=- \beta_j \alpha_i=\delta_{ij} \gamma$ (where $\delta_{ij}$ is the Kronecker symbol) and this, together with \eqref{eq:cohomology-product}, implies the claim.
\end{proof}
If $X$ is a topological space and $X^h$ is its $h$th Cartesian product, we denote by $\Delta \subset X^h$ the big diagonal, namely
\begin{equation*}
\Delta=\{(x_1,\ldots, x_h) \in X^h \, | \, \ x_i = x_j \; \; \textrm{for some}\; \; i \neq j\}.
\end{equation*}

\begin{definition} \label{def:configuration}
The $h$th \emph{ordered configuration space} of $X$ is defined as
\begin{equation} \label{eq:n-conf}
X^h - \Delta=\{(x_1,\ldots, x_h) \in X^h \, | \, \ x_i \neq x_j \;\; \textrm{for all}\; \;i \neq j\}.
\end{equation}
\end{definition}

We will focus on the case $X=\Sigma_b$ and $h=2$, for which the following result holds, see \cite[Corollary 12]{Az15}, \cite{Kr94}, \cite{To93}. 
\begin{proposition} \label{prop:poincare-poly}
The Betti numbers of $\Sigma_b \times \Sigma_b - \Delta$ are
\begin{equation} \label{eq:poincare-poly}
\mathsf{b}_0=1, \quad \mathsf{b}_1=4g, \quad \mathsf{b}_2=4g^2+1, \quad  \mathsf{b}_3=2g, \quad  \mathsf{b}_4=0.
\end{equation}
Moreover, the inclusion map $\iota \colon \Sigma_b \times \Sigma_b - \Delta \to \Sigma_b \times \Sigma_b$ induces isomorphisms
\begin{equation} \label{eq:delta-identifications}
\begin{split}
H^1(\Sigma_b \times \Sigma_b - \Delta, \, \mathbb{Q}) & \simeq H^1(\Sigma_b \times \Sigma_b, \, \mathbb{Q}) \\
H^2(\Sigma_b \times \Sigma_b - \Delta, \, \mathbb{Q}) & \simeq H^2(\Sigma_b \times \Sigma_b, \, \mathbb{Q})/ \langle \delta \rangle \\
H^3(\Sigma_b \times \Sigma_b - \Delta, \, \mathbb{Q}) & \simeq H^3(\Sigma_b \times \Sigma_b, \, \mathbb{Q})/\langle (1 \otimes \alpha_i)\delta, \, (1 \otimes \beta_i) \delta \rangle,
\end{split}
\end{equation}
where $\delta \in H^2(\Sigma_b \times \Sigma_b, \, \mathbb{Q})$ stands for the cohomology class of the diagonal $\Delta$.
\end{proposition}

\subsection{Surface pure braid groups} \label{subsec:braid}

Let $\Sigma_b$ be a smooth projective curve of genus $b \geq 2$ as above, and let  $\mathscr{P} = \{p_1, \ldots, p_h\} \subset \Sigma_b$ be an ordered set of $h$ distinct points. A \emph{pure geometric braid} on $\Sigma_b$ based at $\mathscr{P}$ is a $h$-tuple $(\alpha_1, \ldots, \alpha_h)$ of paths $\alpha_i \colon [0, \, 1] \to \Sigma_b$ such that
\begin{itemize}
\item $\alpha_i(0) = \alpha_i(1)=p_i \quad \textrm{for all }i \in \{1, \ldots, h\}$
\item the points $\alpha_1(t), \ldots, \alpha_h(t) \in \Sigma_b$ are pairwise distinct for all $t \in [0, \, 1]$,
\end{itemize}
see Figure \ref{fig:braid}.
\begin{figure}[H]
\begin{center}
\includegraphics*[totalheight=4.5 cm]{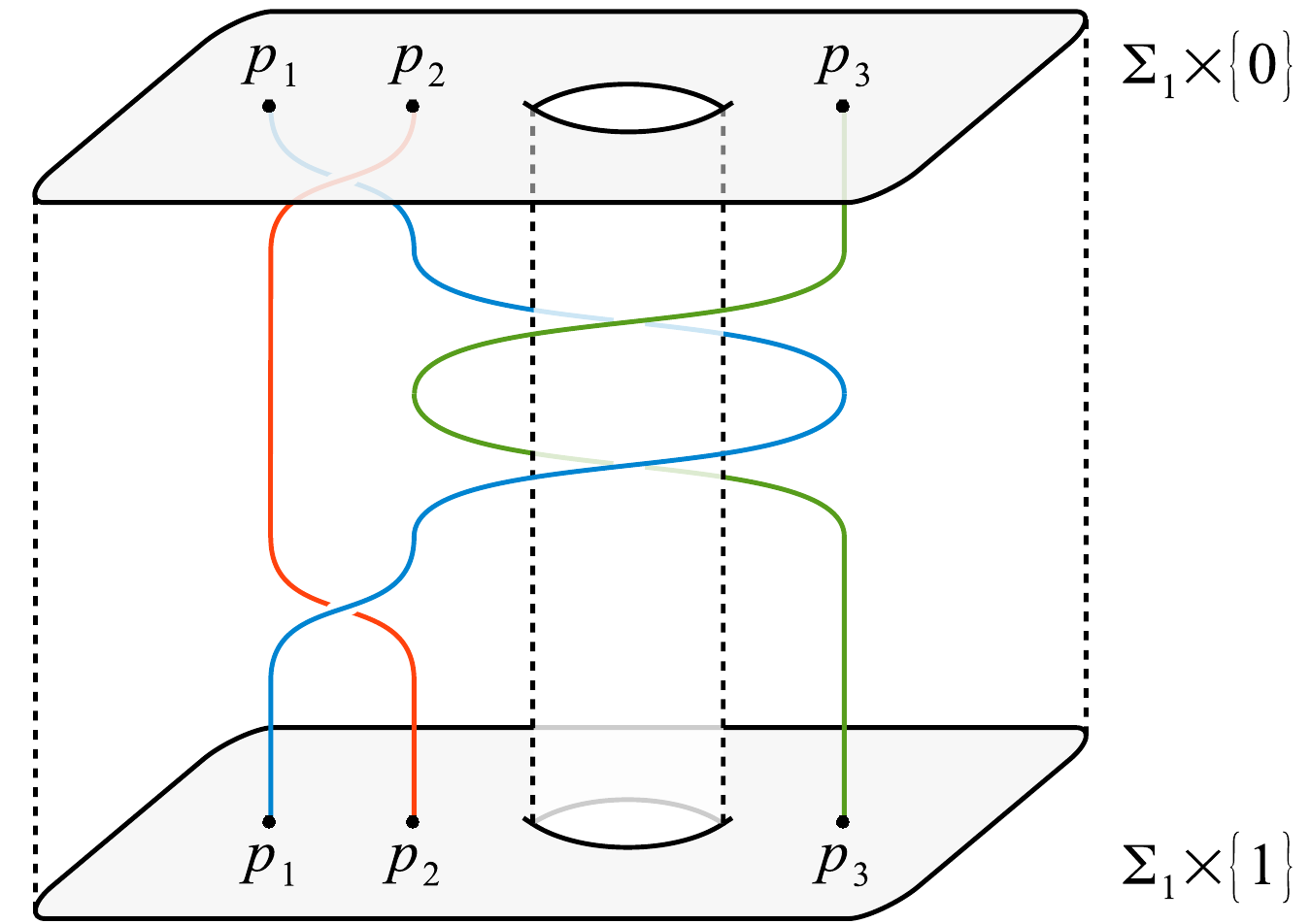}
\caption{A pure braid on $3$ strings} \label{fig:braid}
\end{center}
\end{figure}

\begin{definition} \label{def:braid}
The \emph{pure braid group} on $h$ strings on $\Sigma_b$ is the group $\mathsf{P}_{h}(\Sigma_b)$ whose elements are the pure braids  based at $\mathscr{P}$ and whose operation is the usual concatenation of paths, up to homotopies among braids.
\end{definition}
It can be shown that $\mathsf{P}_{h}(\Sigma_b)$ does not depend on the choice of the set $\mathscr{P}$, and that there is an isomorphism
\begin{equation} \label{eq:iso-braids}
\mathsf{P}_{h}(\Sigma_b) \simeq \pi_1((\Sigma_b)^h - \Delta, \, \mathscr{P}).
\end{equation}

Moreover, the group $\mathsf{P}_{h}(\Sigma_b)$  is finitely presented for all pairs $(b, \, h)$, and explicit presentations can be found in \cite{Bel04, Bir69, GG04, S70}. Again, we will focus on the case $h=2$, referring the reader to \cite[Sections 1-3]{GG04} for a treatment of the general situation.
\begin{proposition}[{\cite[Theorem 1]{GG04}}] \label{prop:split-braid}
Let $p_1, \, p_2 \in \Sigma_b$, with $b \geq 2$. Then the map of pointed topological spaces given by the projection onto the first component 
\begin{equation} \label{eq:proj-first}
(\Sigma_b \times \Sigma_b - \Delta, \, (p_1, \, p_2)\,) \to (\Sigma_b, \, p_1) 
\end{equation}
induces a split short exact sequence of groups
\begin{equation} \label{eq:split-braid}
1 \longrightarrow \pi_1(\Sigma_b - \{p_1\}, \, p_2)  \longrightarrow \mathsf{P}_{2}(\Sigma_b) \longrightarrow \pi_1(\Sigma_b, \, p_1) \longrightarrow 1.
\end{equation}
\end{proposition}
For all $j \in \{1, \ldots, b\}$, let us consider  the elements
\begin{equation} \label{eq:braid-generators}
\rho_{1j}, \; \tau_{1 j}, \; \rho_{2j}, \; \tau_{2 j}
\end{equation}
of $\mathsf{P}_{2}(\Sigma_b)$ represented by the pure braids shown in Figure \ref{fig1}.

\begin{figure}[H]
\begin{center}
\includegraphics*[totalheight=3 cm]{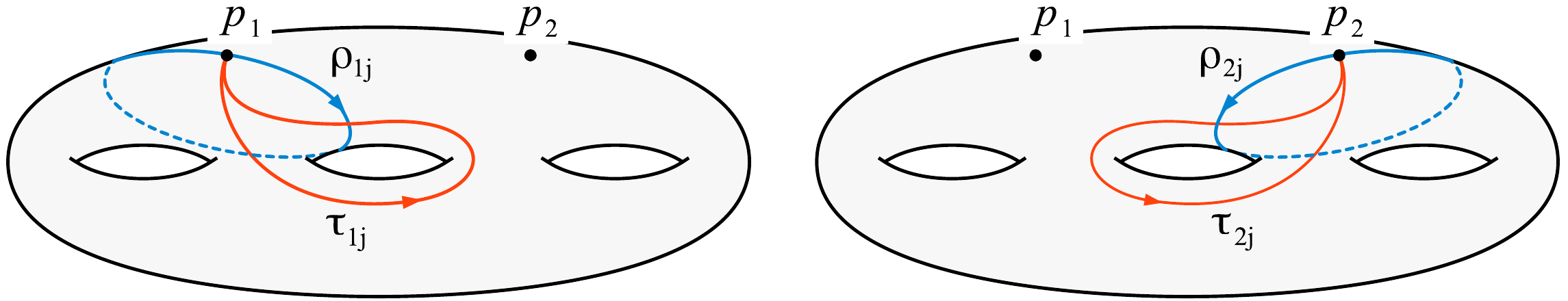}
\caption{The pure braids $\rho_{1j}, \, \tau_{1j}$, $\rho_{2j}, \, \tau_{2j}$ on $\Sigma_b$} \label{fig1}
\end{center}
\end{figure}

If $\ell \neq i$, the path corresponding to $\rho_{ij}$ and $\tau_{ij}$ based at $p_{\ell}$ is the constant path. Moreover, let $A_{12}$ be the pure braid shown in Figure \ref{fig2}.

\begin{figure}[H]
\begin{center}
\includegraphics*[totalheight=1.5 cm]{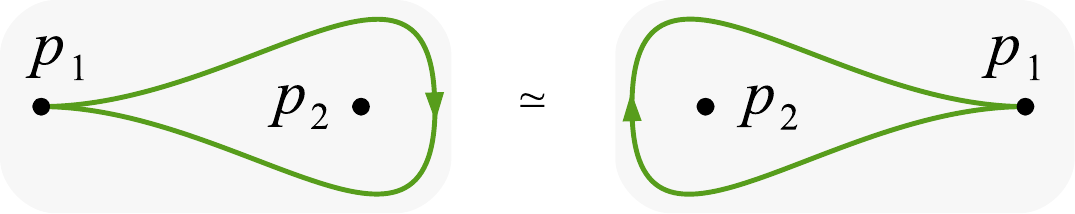}
\caption{The pure braid $A_{12}$ on $\Sigma_b$} \label{fig2}
\end{center}
\end{figure}

The elements
\begin{equation} \label{eq:gens-kernel}
\rho_{21},\ldots, \rho_{2b}, \; \tau_{21},\ldots, \tau_{2b}, \; A_{12}
\end{equation}
can be seen as generators of the kernel $\pi_1(\Sigma_b - \{p_1\}, \, p_2)$ in \eqref{eq:split-braid}, whereas the elements
\begin{equation} \label{eq:gens-quotient}
\rho_{11},\ldots, \rho_{1b}, \; \tau_{11},\ldots, \tau_{1b}
\end{equation}
are lifts of a set of generators of $\pi_1(\Sigma_b, \, p_1)$ via the quotient map $\mathsf{P}_{2}(\Sigma_b) \to \pi_1(\Sigma_b, \, p_1)$, namely, they form a complete system of coset representatives for $\pi_1(\Sigma_b, \, p_1)$. By Proposition \ref{prop:split-braid}, the group $\mathsf{P}_{2}(\Sigma_b)$ is a semi-direct product of the two groups  $\pi_1(\Sigma_b - \{p_1\}, \, p_2)$ and $\pi_1(\Sigma_b, \, p_1)$, whose presentations are both well-known; then, in order to write down a presentation for our braid group, it only remains to specify how the generators in \eqref{eq:gens-quotient} act by conjugation on those in \eqref{eq:gens-kernel}. This information is encoded in  the following result, where the conjugacy relations are expressed by using commutators (i.e., instead of $xyx^{-1}=z$ we write $[x, \, y]=zy^{-1}$).

\begin{theorem}[{\cite[Theorem 7]{GG04}}] \label{thm:presentation-braid}
The group $\mathsf{P}_{2}(\Sigma_b)$ admits the following presentation. \\ \\
\emph{Generators}

$\rho_{1j}, \; \tau_{1j}, \;  \rho_{2j}, \; \tau_{2 j}, \; A_{1 2}, \quad j=1,\ldots, b.$ \\\\
\emph{Relations}
\begin{itemize}
\item \emph{Surface relations:}
\begin{align} \label{eq:presentation-0}
 & [\rho_{1b}^{-1}, \, \tau_{1b}^{-1}] \, \tau_{1b}^{-1} \, [\rho_{1 \,b-1}^{-1}, \, \tau_{1 \,b-1}^{-1}] \, \tau_{1\,b-1}^{-1} \cdots [\rho_{11}^{-1}, \, \tau_{11}^{-1}] \, \tau_{11}^{-1} \, (\tau_{11} \, \tau_{12} \cdots \tau_{1b})=A_{12} \\
& [\rho_{21}^{-1}, \, \tau_{21}] \, \tau_{21} \, [\rho_{22}^{-1}, \, \tau_{22}] \, \tau_{22}\cdots   [\rho_{2b}^{-1}, \, \tau_{2b}] \, \tau_{2b} \, (\tau_{2b}^{-1} \, \tau_{2 \, b-1}^{-1} \cdots \tau_{21}^{-1}) =A_{12}^{-1}
\end{align}
\item \emph{Action of} $\rho_{1j}:$
\begin{align} \label{eq:presentation-1}
[\rho_{1j}, \, \rho_{2k}]& =1  &  \mathrm{if} \; \; j < k \\
[\rho_{1j}, \, \rho_{2j}]& = 1 & \\
[\rho_{1j}, \, \rho_{2k}]& =A_{12}^{-1} \, \rho_{2k}\, \rho_{2j}^{-1} \, A_{12} \, \rho_{2j}\, \rho_{2k}^{-1} \; \;&  \mathrm{if} \;  \; j > k \\
& \\
[\rho_{1j}, \, \tau_{2k}]& =1 & \mathrm{if}\;  \; j < k \\
[\rho_{1j}, \, \tau_{2j}]& = A_{12}^{-1} & \\
[\rho_{1j}, \, \tau_{2k}]& =[A_{12}^{-1}, \, \tau_{2k}]  & \mathrm{if}\;  \; j > k \\
& \\
[\rho_{1j}, \,A_{12}]& =[\rho_{2j}^{-1}, \,A_{12}] &
\end{align}
\item \emph{Action of} $\rho_{1j}^{-1}:$
\begin{align} \label{eq:presentation-2}
[\rho_{1j}^{-1}, \, \rho_{2k}]& =1 & \mathrm{if}\;  \; j < k \\
[\rho_{1j}^{-1}, \, \rho_{2j}]& = 1 & \\
[\rho_{1j}^{-1}, \, \rho_{2k}]& =\rho_{2j} \, A_{12} \, \rho_{2j}^{-1}\, \rho_{2k} \, A_{12}^{-1} \, \rho_{2k}^{-1}\; \; & \mathrm{if}\;  \; j > k  \\
& \\
[\rho_{1j}^{-1}, \, \tau_{2k}]& =1 & \mathrm{if}\;  \; j < k \\
[\rho_{1j}^{-1}, \, \tau_{2j}]& = \rho_{2j}\,A_{12}\,\rho_{2j}^{-1} &  \\
[\rho_{1j}^{-1}, \, \tau_{2k}]& =\rho_{2j}\, A_{12}\, \rho_{2j}^{-1}\, \tau_{2k}\, \rho_{2j}\,A_{12}^{-1}\, \rho_{2j}^{-1}\, \tau_{2k}^{-1} \; \; &  \mathrm{if}\;  \; j > k \\
& \\
[\rho_{1j}^{-1}, \,A_{12}]& =[\rho_{2j}, \,A_{12}] &
\end{align}
\item \emph{Action of} $\tau_{1j}:$
\begin{align} \label{eq:presentation-3}
[\tau_{1j}, \, \rho_{2k}]& =1 & \mathrm{if}\;  \; j < k \\
[\tau_{1j}, \, \rho_{2j}]& = \tau_{2j}^{-1}\, A_{12}\, \tau_{2j} & \\
[\tau_{1j}, \, \rho_{2k}]& =[\tau_{2j}^{-1},\, A_{12}] \; \; & \mathrm{if}  \;\; j > k \\
& \\
[\tau_{1j}, \, \tau_{2k}]& =1 & \mathrm{if}\; \; j < k \\
[\tau_{1j}, \, \tau_{2j}]& = [\tau_{2j}^{-1}, \, A_{12}] & \\
[\tau_{1j}, \, \tau_{2k}]& =\tau_{2j} ^{-1}\,A_{12}\, \tau_{2j}\, A_{12}^{-1}\, \tau_{2k}\,A_{12}\, \tau_{2j} ^{-1}\,A_{12}^{-1}\, \tau_{2j}\,\tau_{2k}^{-1} \; \;  & \mathrm{if}\;  \; j > k \\
&  \\
[\tau_{1j}, \,A_{12}]& =[\tau_{2j}^{-1}, \,A_{12}] &
\end{align}
\item \emph{Action of} $\tau_{1j}^{-1}:$
\begin{align} \label{eq:presentation-4}
[\tau_{1j}^{-1}, \, \rho_{2k}]& =1 & \mathrm{if}\;  \; j < k \\
[\tau_{1j}^{-1}, \, \rho_{2j}]& =  A_{12}^{-1} & \\
[\tau_{1j}^{-1}, \, \rho_{2k}]& =[ A_{12}^{-1}, \, \tau_{2j}] \; \; & \mathrm{if}\;  \; j > k \\
& \\
[\tau_{1j}^{-1}, \, \tau_{2k}]& =1 & \mathrm{if}\;  \; j < k   \\
[\tau_{1j}^{-1}, \, \tau_{2j}]& = [ A_{12}^{-1}, \, \tau_{2j}] & \\
[\tau_{1j}^{-1}, \, \tau_{2k}]& =A_{12}^{-1} \, \tau_{2j} \, A_{12}\,\tau_{2j}^{-1}\, \tau_{2k}\, \tau_{2j}
\,A_{12}^{-1} \, \tau_{2j}^{-1} \, A_{12} \, \tau_{2k}^{-1}
 \; \;  & \mathrm{if}\;  \; j > k \\
&  \\
[\tau_{1j}^{-1}, \,A_{12}]& =[A_{12}^{-1}, \, \tau_{2j}] &
\end{align}
\end{itemize}
\end{theorem}
\begin{remark} \label{rmk:presentation-braid}
Let us make some comments on the previous result.
\begin{enumerate}[label=(\emph{\roman*})]
\item As remarked in \cite[p. 196]{GG04}, we might deduce the action of $\rho_{ij}^{-1}$ and $\tau_{ij}^{-1}$ from that of $\rho_{ij}$ and $\tau_{ij}$, in other words the sets of relations \eqref{eq:presentation-2} and \eqref{eq:presentation-4} are actually redundant. However, since the required computations are rather cumbersome, for the sake of clarity we preferred to explicitly write down all the relations.
\item Tedious but straightforward calculations show that the presentation given in Theorem \ref{thm:presentation-braid} is invariant under the substitutions
\begin{equation*}
A_{12} \longleftrightarrow A_{12}^{-1}, \quad \tau_{1j} \longleftrightarrow \tau_{2\; \; b+1-j}^{-1}, \quad  \rho_{1j} \longleftrightarrow \rho_{2\;\; b+1-j},
\end{equation*}
where $j \in \{1,\ldots, b \}$. These substitutions correspond to the order $2$ automorphism of $\mathsf{P}_2(\Sigma_b)$ induced by the involution of $\Sigma_b$ given by a reflection switching the $j$-th handle with the $(b+1-j)$-th handle for all $j$. Hence we can exchange the roles of $p_1$ and $p_2$ in \eqref{eq:split-braid}, and
see $\mathsf{P}_2(\Sigma_b)$ as the middle term of a split short exact sequence of the form
\begin{equation} \label{eq:split-braid-new}
1 \longrightarrow \pi_1(\Sigma_b - \{p_2\}, \, p_1)  \longrightarrow \mathsf{P}_{2}(\Sigma_b) \longrightarrow \pi_1(\Sigma_b, \, p_2) \longrightarrow 1,
\end{equation}
induced by the projection onto the second component 
\begin{equation} \label{eq:proj-second}
(\Sigma_b \times \Sigma_b - \Delta, \, (p_1, \, p_2)\,) \to (\Sigma_b, \, p_2). 
\end{equation}
Now the elements
\begin{equation} \label{eq:gens-kernel-new}
\rho_{11},\ldots, \rho_{1b}, \; \tau_{11},\ldots, \tau_{1b}, \; A_{12}
\end{equation}
can be seen as generators of the kernel $\pi_1(\Sigma_b - \{p_2\}, \, p_1)$ in \eqref{eq:split-braid-new}, whereas the elements
\begin{equation} \label{eq:gens-quotient-new}
\rho_{21},\ldots, \rho_{2b}, \; \tau_{21},\ldots, \tau_{2b}
\end{equation}
yield a complete system of coset representatives for $\pi_1(\Sigma_b, \, p_2)$.
\item In \cite[p. 190-191]{GG04}, the splitting of \eqref{eq:split-braid} is obtained by first constructing
an explicit geometric section, that in turn yields  the algebraic section $s \colon \pi_1(\Sigma_b, \, p_1) \to \mathsf{P}_2(\Sigma_b)$ given by
\begin{equation*}
s(\rho_{1j})=\rho_{1j}, \quad s(\tau_{1k})=\tau_{1k}, \quad s(\tau_{1b}) = \tau_{1b}A_{12}^{-1}\tau_{2b},
\end{equation*}
for all $j \in \{1, \ldots, b\}$, $k \in \{1, \ldots, b-1\}$.
\item The inclusion map $\iota \colon \Sigma_b \times \Sigma_b - \Delta \to \Sigma_b \times \Sigma_b$ induces a group epimorphism $\iota_* \colon \mathsf{P}_2(\Sigma_b) \to \pi_1(\Sigma_b \times \Sigma_b)$, whose kernel is the normal closure of the subgroup generated by $A_{12}$. Thus, given any group homomorphism $\varphi \colon \mathsf{P}_2(\Sigma_b) \to G$, it factors through $\pi_1(\Sigma_b \times \Sigma_b)$ if and only if $\varphi(A_{12})$ is trivial.
\item In terms of the isomorphism \eqref{eq:iso-braids}, the generator $A_{12}$ corresponds to the homotopy class in $\Sigma_b \times \Sigma_b - \Delta$ of a topological loop in $\Sigma_b \times \Sigma_b$ that ``winds once around $\Delta$''.
\end{enumerate}
\end{remark}

\subsection{Cohomology of $\Sigma_b \times \Sigma_b - \Delta$: arbitrary coefficients} \label{subsec:config-integer}

From Theorem \ref{thm:presentation-braid}  we easily obtain
\begin{equation} \label{eq:1-(co)homology-braid}
\mathbb{Z}^{4b} \simeq \mathsf{P}_{2}(\Sigma_b)^{\mathrm{ab}} \simeq H_1(\Sigma_b \times \Sigma_b - \Delta, \, \mathbb{Z}) \simeq  \mathrm{Hom}\, (H^1(\Sigma_b \times \Sigma_b - \Delta, \, \mathbb{Z}), \, \mathbb{Z}),
\end{equation}
where the last isomorphism follows from the Universal Coefficient Theorem as stated in  \cite[Theorem 5.5.12]{We14}. We are now ready to compute the cohomology groups of the $2$nd configuration space $\Sigma_b \times \Sigma_b - \Delta$ with coefficients in an arbitrary field. The crucial step consists in showing first that all its integral cohomology groups are torsion-free.
\begin{proposition} \label{prop:integral-cohomology-conf}
The cohomology groups of $\Sigma_b \times \Sigma_b - \Delta$ with integral coefficients are as follows$:$
\begin{equation} \label{eq:cohomology-conf-2}
\begin{split}
H^0(\Sigma_b \times \Sigma_b - \Delta, \, \mathbb{Z}) & \simeq \mathbb{Z} \\
H^1(\Sigma_b \times \Sigma_b - \Delta, \, \mathbb{Z}) & \simeq \mathbb{Z}^{4b} \\
H^2(\Sigma_b \times \Sigma_b - \Delta, \, \mathbb{Z}) & \simeq \mathbb{Z}^{4b^2+1} \\
H^3(\Sigma_b \times \Sigma_b - \Delta, \, \mathbb{Z}) & \simeq \mathbb{Z}^{2b} \\
H^4(\Sigma_b \times \Sigma_b - \Delta, \, \mathbb{Z})& =0.
\end{split}
\end{equation}
Furthermore, for all $i \in \{0, \ldots, 4 \}$ we have
\begin{equation} \label{eq:homology=cohomology}
H_i(\Sigma_b \times \Sigma_b - \Delta, \, \mathbb{Z}) \simeq \mathrm{Hom}\,(H^i(\Sigma_b \times \Sigma_b - \Delta, \, \mathbb{Z}), \, \mathbb{Z}).
\end{equation}
\end{proposition}
\begin{proof}
The projection \eqref{eq:proj-first} onto the first component yields a locally trivial topological fibration $f \colon \Sigma_b \times \Sigma_b - \Delta \to \Sigma_b$ with fibre homeomorphic to $\Sigma_b-\{p_1\}$, see \cite[Theorem 1]{FN62}.
There is a commutative diagram of fibrations
\begin{equation}
\begin{split}
   \xymatrix{
       \Sigma_b \times \Sigma_b - \Delta \ar[r]^-{\iota} \ar[d]_{f} & \Sigma_b \times \Sigma_b \ar[d]^-{\pi_1} \\
 \Sigma_b \ar^{\mathrm{id}}[r]  & \Sigma_b}
     \end{split}
\end{equation}
and moreover
\begin{equation*}
H^q(\Sigma_b-\{p_1\}, \, \mathbb{Z}) \simeq H^q(\Sigma_b, \, \mathbb{Z}), \quad q=0, \, 1.
\end{equation*}
Thus, for every point $x_0 \in \Sigma_b$, the monodromy representations
\begin{equation}
\varrho_q \colon \pi_1(\Sigma_b, \, x_0) \to \mathrm{Aut} \left(H^q( f^{-1}(x_0), \, \mathbb{Z}) \right), \quad q=0, \, 1,
\end{equation}
coincide with the corresponding monodromy representations for the product fibration $\pi_1  \colon \Sigma_b \times \Sigma_b \to \Sigma_b$, and so they are trivial.
This in turn implies that the local systems $R^qf_*\mathbb{Z}$ are isomorphic to the constant sheaves $H^q(\Sigma_b-\{p_1\}, \, \mathbb{Z})$ for $q = 0,\, 1$. Now we claim that
\begin{equation} \label{eq:R2}
R^2f_*\mathbb{Z}=0.
\end{equation}
In fact, by local triviality, there exists a contractible open set $U \subset \Sigma_b$ such that
\begin{equation}
f^{-1}(U)=U \times (\Sigma_b-\{p_1\}),
\end{equation}
and by the K{\"u}nneth formula this yields
\begin{equation}
H^2(f^{-1}(U), \, \mathbb{Z}) \simeq H^2(\Sigma_b-\{p_1\}, \, \mathbb{Z})=0,
\end{equation}
the last equality being a consequence of the fact that $\Sigma_b-\{p_1\}$ is a non-compact, real $2$-manifold, see \cite[Corollary 7.12 p. 346]{Bre93} and \cite[Theorem 5.5.12]{We14}.
Recalling that $R^2f_*\mathbb{Z}$ is the sheaf on $\Sigma_b$ associated to the presheaf $U \mapsto H^2(f^{-1}(U), \, \mathbb{Z})$, we obtain \eqref{eq:R2}, as claimed.

Summarizing the computations above, we see that the $E_2$-page of the Leray spectral sequence for $f\colon \Sigma_b \times \Sigma_b - \Delta \to \Sigma_b$, namely
\begin{equation}
E_2^{p, \, q}=H^p(\Sigma_b, \, R^qf_*\mathbb{Z}) \Rightarrow H^{p+q}(\Sigma_b \times \Sigma_b - \Delta, \, \mathbb{Z}),
\end{equation}
is given by
\begin{center}
\begin{tikzpicture}
  \matrix (m) [matrix of math nodes,
    nodes in empty cells,nodes={minimum width=8ex,
    minimum height=10ex,outer sep=-5pt},
    column sep=1ex,row sep=1ex]{
          q      &      &     &     & \\
          1     &  H^1(\Sigma_b, \, \ZZ) &  H^1(\Sigma_b, \, \ZZ) \otimes H^1(\Sigma_b, \, \ZZ)  & H^2(\Sigma_b, \, \ZZ) \otimes H^1(\Sigma_b, \, \ZZ) & \\
          0     &  H^0(\Sigma_b, \, \ZZ)  & H^1(\Sigma_b, \, \ZZ) &  H^2(\Sigma_b, \, \ZZ)  & \\
    \quad\strut &   0  &  1  &  2  & p \strut \\};
\draw[red, stealth-] (m-3-4)  edge node[above] {$d_2^{0, \, 1}$}  (m-2-2);
\draw[thick] (m-1-1.east) -- (m-4-1.east) ;
\draw[thick] (m-4-1.north) -- (m-4-5.north) ;
\end{tikzpicture}
\end{center}
The only differential that is not automatically zero is $d_2^{0, \, 1} \colon H^1(\Sigma_b, \, \ZZ) \to H^2(\Sigma_b, \, \ZZ)$, so we obtain
\begin{equation} \label{eq:leray}
H^{1}(\Sigma_b \times \Sigma_b - \Delta, \, \mathbb{Z}) \simeq \ker d_2^{0, \, 1} \oplus H^1(\Sigma_b, \, \ZZ) \simeq \ker d_2^{0, \, 1} \oplus \mathbb{Z}^{2b}.
\end{equation}
Comparing \eqref{eq:leray} with  \eqref{eq:1-(co)homology-braid}, we deduce $ \ker d_2^{0, \, 1} \simeq \mathbb{Z}^{2b}$, in other words $d_2^{0, \, 1}$ is the zero map, too. Hence the spectral sequence degenerates at the page $E_2$, and \eqref{eq:cohomology-conf-2} follows.

The isomorphisms in \eqref{eq:homology=cohomology} are now a consequence of the Universal Coefficient Theorem, as stated in \cite[Corollary 5.5.19]{We14}, because all integral cohomology groups of $\Sigma_b \times \Sigma_b - \Delta $ are torsion-free (Proposition \ref{prop:integral-cohomology-conf}).
\end{proof}
The computation of the cohomology groups of $\Sigma_b \times \Sigma_b - \Delta$ over an arbitrary field is now straightforward.

\begin{corollary} \label{cor-cohomology-K-conf}
For \emph{every} field $\mathbb{K}$ and for all $i \in \{0, \ldots, 4\}$, there are isomorphisms
\begin{equation} \label{eq:A}
H^i(\Sigma_b \times \Sigma_b - \Delta, \, \mathbb{K}) \simeq H^i(\Sigma_b \times \Sigma_b - \Delta, \, \mathbb{Z}) \otimes \mathbb{K},
\end{equation}
so that the left-hand side of \eqref{eq:A} can be computed by using \eqref{eq:cohomology-conf-2}.
Moreover, the inclusion map $\iota \colon \Sigma_b \times \Sigma_b - \Delta \to \Sigma_b \times \Sigma_b$ induces isomorphisms
\begin{equation} \label{eq:delta-identifications-K}
\begin{split}
H^1(\Sigma_b \times \Sigma_b - \Delta, \, \mathbb{K}) & \simeq H^1(\Sigma_b \times \Sigma_b, \, \mathbb{K}) \\
H^2(\Sigma_b \times \Sigma_b - \Delta, \, \mathbb{K}) & \simeq H^2(\Sigma_b \times \Sigma_b, \, \mathbb{K})/ \langle \delta \rangle \\
H^3(\Sigma_b \times \Sigma_b - \Delta, \, \mathbb{K}) & \simeq H^3(\Sigma_b \times \Sigma_b, \, \mathbb{K})/\langle (1 \otimes \alpha_i)\delta, \, (1 \otimes \beta_i) \delta \rangle. \\
\end{split}
\end{equation}
Finally, as $\mathbb{K}$-vector spaces we have
\begin{equation} \label{eq:cohom-dual}
\begin{split}
H_i(\Sigma_b \times \Sigma_b - \Delta, \, \mathbb{K})  \simeq  H^i(\Sigma_b \times \Sigma_b - \Delta, \, \mathbb{K})^{\vee}.
\end{split}
\end{equation}
\end{corollary}
\begin{proof}
The isomorphism \eqref{eq:A} comes from the Universal Coefficient Theorem as stated in \cite[Theorems 5.3.9]{We14}. The same theorem, applied to \eqref{eq:delta-identifications}, yields  \eqref{eq:delta-identifications-K}. Finally, \eqref{eq:cohom-dual} is an immediate consequence of  \cite[Theorems 5.5.19]{We14}.
\end{proof}

\begin{remark} \label{rem:class-of-diagonal}
The cohomology class $\delta \in H^2(\Sigma_b \times \Sigma_b, \,  \mathbb{K})$ of the diagonal $\Delta \subset \Sigma_b \times \Sigma_b$ is
\begin{equation}
\delta = \gamma \otimes 1 + 1 \otimes \gamma + \sum_{j=1}^b (\beta_j \otimes \alpha_j - \alpha_j \otimes \beta_j),
\end{equation}
see for instance \cite[Theorem 11.11]{MilSt74}, so we can rewrite the last isomorphism in \eqref{eq:delta-identifications-K} as
\begin{equation}
H^3(\Sigma_b \times \Sigma_b - \Delta, \, \mathbb{K}) \simeq H^3(\Sigma_b \times \Sigma_b, \, \mathbb{K})/\langle \alpha_i \otimes \gamma + \gamma \otimes \alpha_i, \, \beta_i \otimes \gamma + \gamma \otimes \beta_i \rangle.
\end{equation}
\end{remark}
We can also prove the analogue of Proposition \ref{prop:cup-product-easy} for the 2nd configuration space of $\Sigma_b$.
\begin{proposition} \label{prop:cup-product}
For \emph{every} field $\mathbb{K}$, the cup-product pairing
\begin{equation} \label{eq:cup-product}
\eta \colon \wedge^2 H^1(\Sigma_b \times \Sigma_b - \Delta, \, \mathbb{K}) \longrightarrow H^2(\Sigma_b \times \Sigma_b - \Delta, \, \mathbb{K})
\end{equation}
is surjective.
\end{proposition}
\begin{proof}
We have a commutative diagram
\begin{equation} \label{dia:cup-product}
\begin{CD}
\wedge^2 H^1(\Sigma_b \times \Sigma_b, \, \mathbb{K})  @> {\xi} >> H^2(\Sigma_b \times \Sigma_b, \, \mathbb{K}) \\
@VV  V  @VV  V\\
\wedge^2 H^1(\Sigma_b \times \Sigma_b - \Delta, \, \mathbb{K}) @> {\eta}>> H^2(\Sigma_b \times \Sigma_b - \Delta, \, \mathbb{K}). \\
\end{CD}
\end{equation}
By Proposition \ref{prop:cup-product-easy}, the map $\xi$ is an epimorphism. Moreover, by Corollary \ref{cor-cohomology-K-conf}, the vertical map on the right is an epimorphism as well, whereas the vertical map on the left is an isomorphism, so the result follows.
\end{proof}

Given a group $\pi$, a $K(\pi, \, 1)$-\emph{space}, also called an \emph{aspherical space}, is a CW-complex $X$ such that $\pi_1(X)=\pi$ and $\pi_i(X)=0$ for all $i >1$. By Whitehead's theorem, this is equivalent to the universal cover of $X$ being contractible, see \cite[Proposition 4.1 p. 342 and Theorem 4.5  p. 346]{Hat02}.

\begin{lemma} \label{lem:kp1}
The $2$\emph{nd} configuration space $\Sigma_b \times \Sigma_b - \Delta$ is a $K(\mathsf{P}_{2}(\Sigma_b), \, 1)$-space.
\end{lemma}
\begin{proof}
For all $i>1$, the homotopy long exact sequence for the fibration $f \colon \Sigma_b \times \Sigma_b - \Delta \to \Sigma_b$ gives
\begin{equation*}
\ldots \to \pi_i(\Sigma_b-\{p_1\}) \to \pi_i(\Sigma_b \times \Sigma_b - \Delta)  \to \pi_i(\Sigma_b) \to \ldots
\end{equation*}
Then, if we prove that both $\Sigma_b-\{p_1\}$ and $\Sigma_b$ are aspherical spaces, we are done.
For $\Sigma_b-\{p_1\}$ this follows from the fact that it has the homotopy type of a bouquet of $2b$ copies of $S^1$, and so its higher homotopy groups vanish, whereas for $\Sigma_b$ it is a consequence of the fact that its universal cover is the Poincar{\'e} upper half-space, which is contractible.
\end{proof}

We can now state the following particular case of the topological interpretation of the cohomology groups of a group $\pi$ as the cohomology of an aspherical space with fundamental group $\pi$.
\begin{proposition} \label{prop:maclane}
For every field $\mathbb{K}$ and for all $i \geq 1$, there are isomorphisms
\begin{equation} \label{eq:maclane}
\begin{split}
H^i(\Sigma_b \times \Sigma_b, \, \mathbb{K}) & \simeq H^i(\pi_1(\Sigma_b \times \Sigma_b), \, \mathbb{K}) \\
H^i(\Sigma_b \times \Sigma_b - \Delta, \, \mathbb{K}) & \simeq H^i(\mathsf{P}_{2}(\Sigma_b), \, \mathbb{K}), \\
\end{split}
\end{equation}
where $\mathbb{K}$ is endowed, as an abelian group, with the trivial $\pi_1(\Sigma_b \times \Sigma_b)$-module structure $($respectively, with the trivial $\mathsf{P}_{2}(\Sigma_b)$-module structure$)$.
\end{proposition}
\begin{proof}
The space $\Sigma_b \times \Sigma_b$ is aspherical and, by Lemma \ref{lem:kp1}, the same is true for
$\Sigma_b \times \Sigma_b - \Delta$. So the claim follows by using \cite[Theorem 11.5 p. 136]{ML95}, together with \eqref{eq:iso-braids} for the second isomorphism.
\end{proof}

\section{Lifting homomorphisms from surface braid groups onto finite Heisenberg groups} \label{sec:rep-braid}

\subsection{Extra-special $p$-groups and finite Heisenberg groups}

The following definition can be found in many textbooks on finite group theory, see for instance \cite[p. 183]{Gor07}, \cite[p. 123]{Is08}. In the sequel, we will use the notation $\mathbb{Z}_p$ for both the group and the field with $p$ elements.

\begin{definition} \label{def:extra-special}
Let $p$ be an odd prime number. A finite $p$-group $G$ is called \emph{extra-special} if its center $Z(G)$ is cyclic of order $p$ and the quotient $V=G/Z(G)$ is a non-trivial, elementary abelian $p$-group.
\end{definition}

An elementary abelian $p$-group is a finite-dimensional vector space over the field $\mathbb{Z}_p$, hence it is of the form $V=(\mathbb{Z}_p)^{\dim V}$ and $G$ fits into a short exact sequence
\begin{equation} \label{eq:extension-extra}
1 \to \mathbb{Z}_p \stackrel{i}{\to} G \stackrel{\pi}{\to} V \to 1.
\end{equation}
Note that, $V$ being abelian, we must have $[G, \, G]=\mathbb{Z}_p$, namely the commutator subgroup of $G$ coincides with its center. Furthermore, since the extension \eqref{eq:extension-extra} is central, it cannot be split, otherwise $G$ would be isomorphic to the direct product of the two abelian groups $\mathbb{Z}_p$ and $V$, which is impossible because $G$ is non-abelian.  As a consequence, the equivalence class of such an extension gives a non-zero element in the second cohomology group  $H^2(V, \, \mathbb{Z}_p)$, where $\mathbb{Z}_p$ has the structure of a trivial $V$-module, see \cite[Sections 9.1--9.3]{Rot02}. We shall now identify this element, following \cite[Section 10]{BC92}.

We first define a function $\epsilon \colon V \to \mathbb{Z}_p$ by taking $\epsilon(v)=w^p$, where $w$ is any element with $\pi(w)=v$. Since commutators are central and of order $p$, we have
\begin{equation}
\epsilon(v_1v_2)=(w_1w_2)^p = w_1^p w_2^p [w_2^{-1}, \, w_1^{-1}]^\frac{p(p-1)}{2}= w_1^p w_2^p = \epsilon(v_1) \epsilon(v_2),
\end{equation}
so that $\epsilon$ is a linear map, namely an element of the dual space $V^{\vee}$.

We now define a bilinear map $\tilde{\omega} \colon V \times V \to \mathbb{Z}_p$ as follows. If $v_1, \, v_2 \in V$, we take any two elements $w_1, \, w_2 \in G$ with $\pi(w_1)=v_1$, $\pi(w_2)=v_2$ and we set
\begin{equation} \label{eq:tilde-omega}
\tilde{\omega}(v_1, \, v_2) =[w_1, \, w_2].
\end{equation}
By definition we have $\tilde{\omega}(v, \, v)=0$, hence $\tilde{\omega}$ is an alternating map, namely an element of the vector space $\mathrm{Alt}^2(V) \simeq \Lambda^2(V^{\vee})$. Assume now that $v \in V$ is such that the linear functional $\tilde{\omega}(v, \, \cdot)$ on $V$ is identically zero, and let $w \in G$ be any element such that $\pi(w)=v$. Then $[w, \, \cdot]$ is identically zero on $G$, in other words $w \in Z(G)= \ker \pi$, and so $v=\pi(w)=0$. Therefore $\tilde{\omega}$ is non-degenerate, so it is a symplectic form on $V$; in particular, this implies that $\dim V$ is even.

The following description of the cohomology algebra $H^*(V, \, \mathbb{Z}_p)$  can be found in \cite[p. 1]{AAG09}.

\begin{proposition} \label{prop:subspace-Heis}
Let $V$ be an elementary abelian $p$-group, and let $\mathbb{Z}_p$ be endowed with the structure of trivial $V$-module. Then there is an isomorphism of graded algebras
\begin{equation} \label{eq:graded-cohomology}
H^*(V, \, \mathbb{Z}_p) \simeq \Lambda(V^{\vee}) \otimes_{\mathbb{Z}_p} S(V^{\vee}),
\end{equation}
where the exterior copy of the dual space $V^{\vee}$ is $H^1(V, \mathbb{Z}_p)$ and the polynomial copy lives in $H^2( V, \, \mathbb{Z}_p);$ specifically, the polynomial copy is the image of the exterior copy under the Bockstein boundary map $\beta \colon H^1(V, \mathbb{Z}_p) \to H^2( V, \, \mathbb{Z}_p)$. In particular, we have
\begin{equation} \label{eq:H2}
H^2(V, \, \mathbb{Z}_p) \simeq \Lambda^2(V^{\vee}) \oplus V^{\vee}
\end{equation}
and, under this identification, the extension class corresponding to the central extension given in \eqref{eq:extension-extra} is $(\tilde{\omega}, \, \epsilon)$.
\end{proposition}

\begin{remark}
Using coordinates, we can rewrite the isomorphism of graded algebras \eqref{eq:graded-cohomology} as
\begin{equation} \label{eq:graded-cohomology-coordinates}
H^*(V, \, \mathbb{Z}_p) \simeq \Lambda(x_1, \ldots, x_n) \otimes _{\mathbb{Z}_p} \mathbb{Z}_p[y_1, \ldots, y_n],
\end{equation}
where $\deg(x_i)=1$, $\deg(y_i)=2$ and the Bockstein map is given by $\beta(x_i)=y_i$. With this notation, a basis for $H^1(V, \, \mathbb{Z}_p)$ is provided by the $x_i$, whereas a basis for $H^2(V, \, \mathbb{Z}_p)$ is provided by the elements $x_i\wedge x_j$, with $i < j$, and by the $y_k$, cf. \cite[p. 225]{BC92}.
\end{remark}

An immediate consequence of the previous discussion is that an extra-special $p$-group has order $p^{2n+1}$ for some $n \in \mathbb{N}$. A classification of such groups, for every $n$ and $p$, can be found in \cite[Section 5.5]{Gor07}; in particular, it follows that the exponent of an extra-special group is either $p$ or $p^2$. We will only present the important examples provided by finite Heisenberg groups, that belong to the former case; to this purpose, let us first define Heisenberg groups in general.

\begin{definition} \label{def:Heis}
Let $\mathbb{K}$ be a field of characteristic $\neq 2$ and let $(V, \, \omega)$ be a symplectic $\mathbb{K}$-vector space. We define the \emph{Heisenberg group} $\mathsf{Heis}(V, \, \omega)$ as the central extension 
\begin{equation} \label{eq:central}
1 \to \mathbb{K} \to \mathsf{Heis}(V, \, \omega) \to V \to 1
\end{equation}
of the additive group $V$ given as follows: the underlying set of $ \mathsf{Heis}(V, \, \omega)$ is $V \times \mathbb{K}$, endowed with the group law
\begin{equation} \label{eq:group-law-Heis}
(v_1, \, t_1)\,(v_2, \, t_2) = \left(v_1+v_2, \, t_1+t_2 + \frac{1}{2} \omega(v_1, \, v_2)\right).
\end{equation}
\end{definition}
By basic linear algebra, all symplectic forms on a $2n$-dimensional vector space $V$ are equivalent to the standard symplectic form $\omega_{\mathrm{st}}$ on $\mathbb{K}^{2n}$; thus, given two symplectic forms $\omega_1$, $\omega_2$ on $V$, the two Heisenberg groups $\mathsf{Heis}(V, \, \omega_1)$, $\mathsf{Heis}(V, \, \omega_2)$ are isomorphic, because  they are both isomorphic to $\mathsf{Heis}(\mathbb{K}^{2n}, \, \omega_{\mathrm{st}})$.  Moreover, the latter group is in turn isomorphic to the \emph{matrix Heisenberg group} $\mathsf{H}_{2n+1}(\mathbb{K})$, the subgroup of $\mathrm{GL}_{n+2}(\mathbb{K})$ consisting of matrices with $1$ along the diagonal and $0$ elsewhere, except for the top row and rightmost column, namely
\begin{equation} \label{eq:heis-matrices}
\mathsf{H}_{2n+1}(\mathbb{K}) = \left\{
\begin{pmatrix} 1 & \mathbf{x} & z\\
 {}^t\mathbf{0} &  I_n & {}^t\mathbf{y}  \\
 0 & \mathbf{0} & 1 \\
\end{pmatrix} \; \; \bigg|  \; \; \mathbf{x}, \, \mathbf{y} \in \mathbb{K}^n, \, z \in \mathbb{K}
  \right\}.
\end{equation}
An explicit isomorphism  is obtained by sending a matrix as above to the pair
\begin{equation} \label{eq:iso-heisenberg-matric}
(v, \, t)=\left(( \mathbf{x}, \,  \mathbf{y}), \; z-\frac{1}{2} \mathbf{x} \cdot \mathbf{y} \right) \in \mathsf{Heis}(\mathbb{K}^{2n}, \, \omega_{\mathrm{st}}),
\end{equation}
where $( \mathbf{x}, \,  \mathbf{y}) \in \mathbb{K}^n \times \mathbb{K}^n = \mathbb{K}^{2n}$ and $\cdot$ denotes the standard scalar product in $\mathbb{K}^n$. The subscript of $\mathsf{H}_{2n+1}(\mathbb{K})$ stands for its $2n+1$ standard generators, namely those matrices having all the entries not on the diagonal equal to $0$, except one that is equal to $1$. In particular, the center of $\mathsf{H}_{2n+1}(\mathbb{K})$ is the copy of $\mathbb{K}$ generated by the matrix with $ \mathbf{x}= \mathbf{y} = \mathbf{0}$, $z=1$.

\begin{remark} \label{rmk:two-descriptions-heisenberg}
The description of the Heisenberg group as $\mathsf{H}_{2n+1}(\mathbb{K})$ has the advantage of working in every characteristic, but, unlike $\mathsf{Heis}(V, \, \omega)$, it is not coordinate-free. On the other hand, neither $\mathsf{Heis}(V, \, \omega)$ nor the isomorphism \eqref{eq:iso-heisenberg-matric} make sense in characteristic $2$, because of the presence of the coefficient $1/2$. When $\textrm{char}(\mathbb{K})=2$, we might be tempted to define $\mathsf{Heis}(V, \omega)$ by using the group law
\begin{equation} \label{eq:group-law-Heis-char-2}
(v_1, \, t_1)\,(v_2, \, t_2) = \left(v_1+v_2, \, t_1+t_2 + \omega(v_1, \, v_2)\right)
\end{equation}
on $V \times \mathbb{K}$. Although this makes sense, nevertheless the group obtained in this way is \emph{abelian} (and so it is not isomorphic to $\mathsf{H}_{2n+1}(\mathbb{K})$). In particular, it is not suitable for the construction of  Kodaira fibrations as finite Heisenberg covers of $\Sigma_b \times \Sigma_b$ that we will present in Section \ref{sec:Kodaira}, see Remark \ref{rmk:non-abelian-G}.
\end{remark}

Let us now show that, if $\mathbb{K}= \mathbb{Z}_p$, then $\mathsf{Heis}(V, \, \omega)$ is a finite $p$-group, which is in fact an extra-special one.

\begin{proposition} \label{prop:Heisenberg-is-extra-special}
Let $p$ be an odd prime number and $\mathbb{K}=\mathbb{Z}_p$. If $\dim \, V=2n$, then $\mathsf{Heis}(V, \, \omega)$ is an extra-special $p$-group of order $p^{2n+1}$ and exponent $p$. Moreover, the symplectic form $\tilde{\omega}$ on $V$, defined in \eqref{eq:tilde-omega}, can be naturally identified with $\omega$. Finally,
fixed any $\mathbb{Z}_p$-basis $\{v_1, \ldots, v_{2n}\}$ of $V$, the group $\mathsf{Heis}(V, \, \omega)$ admits the following presentation.\\

\noindent \emph{Generators:}
\begin{equation} \label{eq:gen-heis}
\mathsf{v}_1, \ldots, \mathsf{v}_{2n}, \; \mathsf{z},
\end{equation}
where 
\begin{equation} \label{eq:yet-again-generators}
\mathsf{v}_1=(v_1, \, 0), \ldots, \mathsf{v}_{2n}=(v_{2n}, \, 0), \quad \mathsf{z}=(0, \,1)
\end{equation}
\noindent \emph{Relations:}
\begin{equation} \label{eq:rel-heis}
\begin{split}
\mathsf{v}_1^p & = \ldots =\mathsf{v}_{2n}^p=\mathsf{z}^p=1 \\
[\mathsf{v}_{1}, \, \mathsf{z}]& = \ldots = [\mathsf{v}_{2n}, \, \mathsf{z}]= 1 \\
[\mathsf{v}_i, \, \mathsf{v}_j]& =\mathsf{z}^{\omega(v_i, \, v_j)}
\end{split}
\end{equation}
where $i, \, j \in \{1, \ldots, 2n\}$ and the exponent in $\mathsf{z}^{\omega(v_i, \, v_j)}$ stands for any representative in $\mathbb{Z}$ of $\omega(v_i, \, v_j) \in \mathbb{Z}_p$.

\end{proposition}
\begin{proof}
First, we identify the subgroup $\{0\} \times \mathbb{Z}_p$ of $V \times \mathbb{Z}_p$ with $\mathbb{Z}_p$, and we check that it is the center of $\mathsf{Heis}(V, \, \omega)$. In fact, using the group law \eqref{eq:group-law-Heis}, we see that an element $g=(v, \, t) \in \mathsf{Heis}(V, \, \omega)$ is in the center if and only if
\begin{equation*}
\omega(v, \, \cdot)= \omega(\cdot, \, v),
\end{equation*}
that is, if and only if the linear functional $\omega(v, \, \cdot)$ on $V$ is identically zero. This in turn implies $v=0$ because $\omega$ is non-degenerate, and so $g=(0, \, t) \in \mathbb{Z}_p$, as desired.

Next, for every $g=(v, \, t) \in \mathsf{Heis}(V, \, \omega)$, using
the fact that $\omega(v, \, v)=0$, we obtain  $g^p=(pv, \, pt)=(0, \, 0)$, and this shows that $\mathsf{Heis}(V, \, \omega)$ has exponent $p$.

Finally, take the symplectic form $\tilde{\omega}$ on $V$ given by  \eqref{eq:tilde-omega}. Given two elements $v, \, v' \in V$, we can lift them to the elements $x=(v, \, 0)$, $x'=(v', \, 0)$ in $\mathsf{Heis}(V, \, \omega)$, and by definition we obtain
\begin{equation}
\begin{split}
\tilde{\omega}(v, \, v') = [x, \, x'] & = (v, \, 0)\,(v', \, 0)\,(v, \, 0)^{-1}\,(v', \, 0)^{-1} \\
& = (v, \, 0)\,(v', \, 0)\,(-v, \, 0)\,(-v', \, 0), \\
& = \left(v+v', \, \frac{1}{2}\omega(v, \, v') \right) \, \left(-v-v', \, \frac{1}{2}\omega(-v, \, -v') \right) \\
& = \left(v+v'-v-v', \,  \frac{1}{2}\omega(v, \, v')+ \frac{1}{2}\omega(-v, \, -v') +  \frac{1}{2} \omega(v+v', \, -v-v') \right)\\
& = (0, \, \omega(v, \, v')) \in \mathbb{Z}_p,
\end{split}
\end{equation}
so we can naturally identify $\tilde{\omega}$ with $\omega$. Checking that \eqref{eq:rel-heis} is a presentation for $\mathsf{Heis}(V, \, \omega)$ is now straightforward.
\end{proof}

\begin{remark} \label{rmk:heis-2-group}
If $p$ is an odd prime, Proposition \ref{prop:Heisenberg-is-extra-special} and the isomorphism \eqref{eq:iso-heisenberg-matric} show that the group $\mathsf{H}_{2n+1}(\mathbb{Z}_p)$ is an extra-special $p$-group of order $p^{2n+1}$ and exponent $p$. Moreover, when $(V, \, \omega)  = ((\mathbb{Z}_p)^{2n}, \, \omega_{\mathrm{st}})$ and $\{v_1, \ldots, v_{2n}\}$ is the standard basis of $V$, the generators in \eqref{eq:yet-again-generators} are the images of the standard generators of $\mathsf{H}_{2n+1}(\mathbb{Z}_p)$  via the isomorphism given by \eqref{eq:iso-heisenberg-matric}. If instead $p=2$, the group $\mathsf{H}_{2n+1}(\mathbb{Z}_2)$ is of order $2^{2n+1}$, but its exponent is $4$; for instance,  $\mathsf{H}_{3}(\mathbb{Z}_2)$ is isomorphic to the dihedral group $\mathsf{D}_8$ with $8$ elements.
\end{remark}

\begin{remark} \label{rmk:extension-class-is-omega}
If $\mathbb{K}=\mathbb{Z}_p$, the cohomology class in $H^2(V, \, \mathbb{Z}_p)$ that
corresponds to the Heisenberg extension  \eqref{eq:central} is given by the symplectic form $\omega$. In fact,  by construction, the linear functional $\epsilon \colon V \to \mathbb{Z}_p$ is identically zero, because $\mathsf{Heis}(V, \, \omega)$ has exponent $p$. So our extension class coincides with $\tilde{\omega}$ (Proposition \ref{prop:subspace-Heis}), that in turn can be identified with $\omega$ (Proposition \ref{prop:Heisenberg-is-extra-special}). Actually, $\omega \colon V \times V \to \mathbb{Z}_p$ gives an explicit $2$-cocycle corresponding to the extension; in fact, by bilinearity, for all $v_1, \, v_2, \, v_3 \in V$ we have
\begin{equation*}
\omega(v_1, \, v_2+v_3)-\omega(v_1+v_2, \, v_3) + \omega(v_2, \, v_3)-\omega(v_1, v_2)=0,
\end{equation*}
which is precisely the cocycle identity for the trivial $V$-action on $\mathbb{Z}_p$, cf. \cite[Theorem 9.17]{Rot02}.
\end{remark}

\begin{remark} \label{rmk: degenerate 2-form}
The construction of the Heisenberg group given in Definition \ref{def:Heis} can be generalized to the situation where the alternating form $\omega \colon V \times V \to \mathbb{K}$ is degenerate. In fact, denoting by $V_0$ the kernel of $\omega$, we see that the set $V \times \mathbb{K}$, endowed with the operation \eqref{eq:group-law-Heis}, is a group whose center equals $V_0 \times \mathbb{K}$. Denoting this group still by $\mathsf{Heis}(V, \, \omega)$, the central extension \eqref{eq:central} fits into a commutative diagram
\begin{equation} \label{dia:Heis-V-W}
\begin{CD}
1 @>>> \mathbb{Z}_p  @>>> \mathsf{Heis}(V, \, \omega)  @>>> V @>>> 1\\
@.   @| @VVV  @VVV  \\
1 @>>> \mathbb{Z}_p @>>> \mathsf{Heis}(W, \, \omega) @>>> W @>>>  1,
\end{CD}
\end{equation}
 where we set $W=V/V_0$ and we continue to write $\omega$ for the symplectic form induced by $\omega$ on $W$. If $\dim V = 2n$ and $\dim V_0 = 2 n_0$, with $n_0 \leq n$, then  $\mathsf{Heis}(W, \, \omega)$ is a genuine Heisenberg group in the sense of Definition \ref{def:Heis}, in particular it is an extra-special $p$-group of order $p^{2(n-n_0)+1}$ and exponent $p$, see Proposition \ref{prop:Heisenberg-is-extra-special}.  Applying the snake lemma to diagram \eqref{dia:Heis-V-W}, we can check that the group homomorphism $ \mathsf{Heis}(V, \, \omega) \to \mathsf{Heis}(W, \, \omega)$ is surjective, with kernel equal to $V_0$. In fact, at the level of sets it is given by the map  $V \times \mathbb{K} \to W \times \mathbb{K}$, where we have the quotient by $V_0$ on the first component and    the identity on the second one.
\end{remark}

\subsection{Group cohomology and lifting of homomorphisms}

Suppose that we have a short exact sequence of groups 
\begin{equation} \label{eq:SES}
1 \to K \to G \to V \to 1 
\end{equation}
with abelian kernel, and a group homomorphism $\phi \colon V'\to V$. Taking the fibred product construction as in 
\cite[p. 72]{DHW12}, we obtain a commutative diagram with exact rows
 \begin{equation} \label{diag:lifting-splitting}
\begin{split}
\xymatrix{
1 \ar[r] & K \ar[r] \ar@{=}[d] & G'
\ar[r] \ar[d] & V' \ar[r] 
 \ar[d]^{\phi}  & 1\\
1 \ar[r] & K \ar[r]  & G \ar[r] &
V
\ar[r] & 1.   }
\end{split}
\end{equation}
Writing $u \in H^2(V, \, K)$ and $u' \in H^2(V', \, K)$ for the extension class of the bottom and the top row, respectively, one can check that $u'$ is the image of $u$ via the induced map $\phi^* \colon H^2(V, \, K) \to H^2(V', \, K)$, where $K$ is a $V'$-module via $\phi$. From this, we get the following useful cohomological lifting criterion.
\begin{proposition} \label{prop:pull-back-extension}
The group homomorphism  $\phi \colon V' \to V$ admits a lifting $\varphi \colon V' \to G$  if and only if $\phi^*u=0 \in H^2(V', \, K)$.
\end{proposition}
\begin{proof}
By the definition of fibred product, the liftings $\varphi \colon V' \to G$ of $\phi \colon V' \to V$ are in bijective correspondence to the splittings $s \colon V'\to G'$ of the top sequence in  \eqref{diag:lifting-splitting}, and such a sequence splits if and only if the extension class $u'=\phi^*u$ is zero. 
\end{proof}

\subsection{Liftings onto finite Heisenberg groups} \label{subsec:liftings-onto-Heis}

We will now apply the previous results to the following framework. We set
\begin{equation} \label{eq:setting-V}
K=\mathbb{Z}_p, \quad  V=H_1(\Sigma_b \times \Sigma_b - \Delta, \, \mathbb{Z}_p) \simeq (\mathbb{Z}_p)^{4b}
\end{equation}
and we put a symplectic form $\omega$ on $V$. Correspondingly, we can consider the Heisenberg extension \eqref{eq:central}, whose middle term
$\mathsf{Heis}(V, \, \omega)$ is an extra-special $p$-group of order $p^{4b+1}$ (Proposition \ref{prop:Heisenberg-is-extra-special}).
Furthermore, we denote by
\begin{equation*}
\phi \colon \mathsf{P}_2(\Sigma_b) \longrightarrow  V
\end{equation*}
the surjective group homomorphism given by the composition of the reduction mod $p$ map $H_1(\Sigma_b \times \Sigma_b - \Delta, \, \mathbb{Z}) \to V$ with the abelianization map $\mathsf{P}_2(\Sigma_b) \longrightarrow  H_1(\Sigma_b \times \Sigma_b - \Delta, \, \mathbb{Z})$. Using Proposition \ref{prop:pull-back-extension}, we immediately obtain
\begin{corollary} \label{cor:lifting-Heisenberg}
Let us consider the diagram
\begin{equation} \label{diag:lifting-splitting-heisenberg}
\begin{split}
\xymatrix{
  &  & & \mathsf{P}_2(\Sigma_b) \ar[d]^{\phi} \ar@[red]@{-->}[dl]_{\color{red}{\varphi_{\omega}}} & \\
1 \ar[r] & \mathbb{Z}_p \ar[r]  & \mathsf{Heis}(V, \, \omega) \ar[r]  &
V
\ar[r] & 1, }
\end{split}
\end{equation}
and denote by $u \in H^2(V, \, \mathbb{Z}_p)$ the cohomology class corresponding to the bottom Heisenberg extension. Then a lifting $\varphi_{\omega} \colon \mathsf{P}_2(\Sigma_b) \longrightarrow \mathsf{Heis}(V, \, \omega)$ of $\phi \colon \mathsf{P}_2(\Sigma_b) \to V$  exists if and only if  $\phi^*u=0 \in H^2(\mathsf{P}_2(\Sigma_b), \, \mathbb{Z}_p)$.
\end{corollary}
Now we give an interpretation of the cohomological condition $\phi^* u=0$ in terms of the symplectic form $\omega$. Since
\begin{equation*}
V^{\vee} \simeq H^1(\Sigma_b \times \Sigma_b - \Delta, \, \mathbb{Z}_p)\simeq H^1(\Sigma_b \times \Sigma_b, \, \mathbb{Z}_p),
\end{equation*}
see \eqref{eq:delta-identifications-K} and \eqref{eq:cohom-dual}, there is a commutative diagram
\begin{equation} \label{dia:cup-product-1}
\xymatrix{
\mathrm{Alt}^2(V) \simeq \wedge^2 V^{\vee}   \ar@{->}[r]^{\xi} \ar@{->}[rd]_{\eta}
& H^2(\Sigma_b \times \Sigma_b, \, \mathbb{Z}_p)  \ar@{->}[d]
\\
& H^2(\Sigma_b \times \Sigma_b - \Delta, \, \mathbb{Z}_p)}
\end{equation}
where the vertical map is the quotient by the $1$-dimensional vector subspace of $H^2(\Sigma_b \times \Sigma_b, \, \mathbb{Z}_p)$ generated by the class $\delta$ of the diagonal (see again \eqref{eq:delta-identifications-K}), whereas $\eta$ and $\xi$ stand for the cup product maps. Note that here, with slight abuse of notation, we are identifying $\xi$ and its composition with the isomorphism $\wedge^2 V^{\vee} \to \wedge^2 H^1(\Sigma_b \times \Sigma_b, \, \mathbb{Z}_p)$, cf. diagram \eqref{dia:cup-product}.

\begin{proposition} \label{prop:interpretation-obstruction}
The obstruction class $\phi^*u \in H^2(\mathsf{P}_2(\Sigma_b), \, \mathbb{Z}_p)$ can be naturally interpreted as the image $\eta(\omega) \in H^2(\Sigma_b \times \Sigma_b - \Delta, \, \mathbb{Z}_p)$ of the symplectic form $\omega \in\mathrm{Alt}^2(V)$ via the cup-product map $\eta$. 
\end{proposition}
\begin{proof}
There is a commutative diagram
\begin{equation}
\begin{split}
\xymatrix{
\wedge^2V^{\vee} \oplus V^{\vee} \ar[r]^{\simeq} & H^2(V, \, \mathbb{Z}_p) \ar[d] \ar[r]^-{\phi^*} & H^2(\mathsf{P}_2(\Sigma_b), \, \mathbb{Z}_p) \ar[d]^{\simeq} \\
& \mathrm{Alt}^2(V) \simeq \wedge^2V^{\vee} \ar@{->}[r]^-{\eta} & H^2(\Sigma_b \times \Sigma_b - \Delta, \, \mathbb{Z}_p),
}
\end{split}
\end{equation}
where the isomorphism on the left is \eqref{eq:H2}, the vertical map  on the left is the projection onto the first summand and the vertical map on the right is the second isomorphism in \eqref{eq:maclane}. By Remark \ref{rmk:extension-class-is-omega}, the extension class $u \in H^2(V, \, \mathbb{Z}_p)$ can be naturally identified with $\omega \in \mathrm{Alt}^2(V)$, so the claim follows.
\end{proof}

As a consequence, we obtain the following lifting criterion, that can be seen as the main result of this section.

\begin{theorem} \label{thm:interpretation-obstruction}
The following holds.
\begin{itemize}
\item[$\boldsymbol{(1)}$] A lifting $\varphi_{\omega} \colon \mathsf{P}_2(\Sigma_b) \longrightarrow \mathsf{Heis}(V, \, \omega)$ of $\phi \colon \mathsf{P}_2(\Sigma_b) \to V$  exists if and only if $\eta(\omega)=0$.
\item[$\boldsymbol{(2)}$] If $\varphi_{\omega}$ as in $\boldsymbol{(1)}$ exists, then $\varphi_{\omega}(A_{12})$ has order $p$ if and only if $\xi(\omega)$ is a non-zero integer multiple of the diagonal class
$\delta \in H^2(\Sigma_b \times \Sigma_b, \, \mathbb{Z}_p)$. In this case, $\varphi_{\omega}$ is necessarily surjective.
\end{itemize}
\end{theorem}
\begin{proof}
Part $\boldsymbol{(1)}$ follows from Corollary \ref{cor:lifting-Heisenberg} and Proposition \ref{prop:interpretation-obstruction}, so it only remains to show part $\boldsymbol{(2)}$. We first observe that we have $\phi(A_{12})=0$: in fact, the element $A_{12}$ is a commutator of $\mathsf{P}_2(\Sigma_b)$, see the fifth relation in \eqref{eq:presentation-1}, and, by construction, $\phi$ factors through the abelianization map of
$\mathsf{P}_2(\Sigma_b)$. By Remark \ref{rmk:presentation-braid} $(iv)$, this implies that $\phi$ factors through the group epimorphism
$\iota_* \colon \mathsf{P}_2(\Sigma_b) \to \pi_1(\Sigma_b \times \Sigma_b)$, that is, there exists a group homomorphism $\bar{\phi} \colon \pi_1(\Sigma_b \times \Sigma_b) \to V$ and a commutative diagram
\begin{equation}
\begin{split}
\xymatrix{
  &  & \pi_1(\Sigma_b \times \Sigma_b)  \ar@{-->}@[red][d]_{\color{red}{\bar{\varphi}_{\omega}}} \ar[dr]_{\bar{\phi}} & \mathsf{P}_2(\Sigma_b) \ar[l]_{\quad \quad \iota_*}  \ar[d]^{\phi} 
  & \\
1 \ar[r] & \mathbb{Z}_p \ar[r]  & \mathsf{Heis}(V, \, \omega) \ar[r]  &
V
\ar[r] & 1. }
\end{split}
\end{equation}
The same argument as in the proofs of Corollary \ref{cor:lifting-Heisenberg} and Proposition \ref{prop:interpretation-obstruction}, together with the first isomorphism in \eqref{eq:maclane}, shows that the pull-back $\bar{\phi}^*u \in H^2(\pi_1(\Sigma_b \times \Sigma_b), \, \mathbb{Z}_p)$  can be naturally interpreted as the image $\xi(\omega) \in H^2(\Sigma_b \times \Sigma_b, \, \mathbb{Z}_p)$ of the symplectic form $\omega \in \mathrm{Alt}^2(V)$ via the cup-product map $\xi$, and that we have $\bar{\phi}^*u =0$ if and only if $\bar{\phi}$ can be lifted to a  group homomorphism $\bar{\varphi}_{\omega} \colon  \pi_1(\Sigma_b \times \Sigma_b) \to \mathsf{Heis}(V, \, \omega)$.

On the other hand, the existence of $\bar{\varphi}_{\omega}$ means that $\varphi_{\omega}  \colon \mathsf{P}_2(\Sigma_b) \longrightarrow \mathsf{Heis}(V, \, \omega)$ factors through $\iota_*$ and, again by Remark \ref{rmk:presentation-braid} $(iv)$, this happens if and only if $\varphi_{\omega}(A_{12})$ is trivial.

Summing up, the element $\varphi_{\omega}(A_{12}) \in \mathsf{Heis}(V, \, \omega)$ is non-trivial (or, equivalently, it has order $p$) if and only if $\xi(\omega) \neq 0$. Since we are assuming $\eta(\omega)=0$, this in turn implies that $\xi(\omega)$ is a non-zero element in the kernel of $H^2(\Sigma_b \times \Sigma_b, \, \mathbb{Z}_p) \to H^2(\Sigma_b \times \Sigma_b - \Delta, \, \mathbb{Z}_p)$, namely a non-zero integer multiple of the diagonal class $\delta$.

Finally, being $A_{12}$ a commutator of $\mathsf{P}_2(\Sigma_b)$  implies that $\varphi_{\omega}(A_{12})$ is a commutator of $\mathsf{Heis}(V, \, \omega)$, that is, an element of the center $\mathbb{Z}_p$. Then, if such an element is non-trivial, the fact that $\phi \colon \mathsf{P}_2(\Sigma_b) \longrightarrow  V$ is surjective implies that the same is true for its lifting $\varphi_{\omega} \colon \mathsf{P}_2(\Sigma_b) \longrightarrow \mathsf{Heis}(V, \, \omega)$.
\end{proof}
We can therefore give the following definition, whose geometrical motivation  will become clear in Section \ref{sec:Kodaira}, see in particular Theorem \ref{thm:double-Kodaira-from-omega}.

\begin{definition} \label{def:heisenberg-type}
Set $V=H_1(\Sigma_b \times \Sigma_b - \Delta, \, \mathbb{Z}_p)$ and let $\omega \in \mathrm{Alt}^2(V)$ be a symplectic form on $V$. We will say that $\omega$ is of \emph{Heisenberg type} if it satisfies one of the following equivalent conditions$:$
\begin{itemize}
\item[$\boldsymbol{(1)}$]there exists a surjective lifting $\varphi_{\omega} \colon \mathsf{P}_2(\Sigma_b) \longrightarrow \mathsf{Heis}(V, \, \omega)$ of $\phi \colon \mathsf{P}_2(\Sigma_b) \longrightarrow  V,$ such that $\varphi_{\omega}(A_{12})$ has order $p;$
\item[$\boldsymbol{(2)}$] we have $\eta(\omega)=0$ and $\xi(\omega) \neq 0;$
\item[$\boldsymbol{(3)}$] $\xi(\omega)$ is a non-zero integer multiple of the diagonal class $\delta \in H^2(\Sigma_b \times \Sigma_b, \, \mathbb{Z}_p)$.
\end{itemize}
\end{definition}

Now we want to provide explicit examples of symplectic forms of Heisenberg type. Notice that, by the surjectivity of $\xi$ and $\eta$ (Propositions \ref{prop:cup-product-easy} and \ref{prop:cup-product}), the total number of alternating forms in $\ker \eta - \ker \xi$ equals
\begin{equation*}
 p^{4b^2-2b-1}- p^{4b^2-2b-2}= p^{4b^2-2b-2}(p-1)>0,
\end{equation*}
hence it only remains to exhibit some of these forms that are non-degenerate. We denote again by $\alpha_1, \ldots, \alpha_b$, $\beta_1, \ldots, \beta_b$ the images in $H^1(\Sigma_b, \, \mathbb{Z}_p)=H^1(\Sigma_b, \, \mathbb{Z}) \otimes \mathbb{Z}_p$ of the elements of the symplectic basis of $H^1(\Sigma_b, \, \mathbb{Z})$ given in \eqref{eq:symplectic}, and we choose for $V$ the ordered basis
\begin{equation} \label{eq:ordered-basis-V}
{r}_{11}, \; {t}_{11}, \ldots, r_{1b}, \; {t}_{1b}, \; \;  {r}_{21}, \; {t}_{21}, \ldots, {r}_{2b}, \; {t}_{2b},
\end{equation}
where, under the isomorphism $V \simeq H_1(\Sigma_b \times \Sigma_b, \, \mathbb{Z}_p)$ induced by the inclusion $\iota \colon \Sigma_b \times \Sigma_b - \Delta \to \Sigma_b \times \Sigma_b$, the elements ${r}_{1j}$, ${t}_{1j}$, ${r}_{2j}$, ${t}_{2j} \in V$ are the  duals of the elements $\alpha_j \otimes 1$, $\beta_j \otimes 1$, $1 \otimes \alpha_j$, $1 \otimes \beta_j \in H^1(\Sigma_b \times \Sigma_b, \, \mathbb{Z}_p) \simeq H^1(\Sigma_b, \, \mathbb{Z}_p) \otimes H^1(\Sigma_b, \, \mathbb{Z}_p)$, respectively.

Now we take non-zero scalars $\lambda_1, \ldots, \lambda_b$, $\mu_1, \ldots, \mu_b \in \mathbb{Z}_p$ such that
\begin{equation} \label{eq:lambda-mu}
\sum_{j=1}^b \lambda_j = \sum_{j=1}^b \mu_j =1,
\end{equation}
and we consider the alternating form $\omega \colon V \times V \to \mathbb{Z}_p$ defined on the elements of the basis of $V$ as follows: for all $j \in \{1, \ldots, b\}$ we set
\begin{equation}
\begin{split} \label{eq:symplectic-Kodaira}
\omega({r}_{1j}, \, {t}_{1j}) = - \omega({t}_{1j}, \, {r}_{1j}) & = \lambda_j  \\
\omega({r}_{2j}, \, {t}_{2j}) = - \omega({t}_{2j}, \, {r}_{2j}) & = \mu_j    \\
\omega({r}_{1j}, \, {t}_{2j}) = - \omega({t}_{2j}, \, \mathsf{r}_{1j}) & = -1  \\
\omega({r}_{2j}, \, {t}_{1j}) = - \omega({t}_{1j}, \, {r}_{2j}) & = -1,
\end{split}
\end{equation}
whereas the remaining values are zero. So, with respect to the ordered basis \eqref{eq:ordered-basis-V}, the matrix representing $\omega$ is
\begin{equation} \label{eq:omega}
\Omega_b=
\begin{pmatrix}
L_b & J_b \\
J_b & M_b
\end{pmatrix} \in \mathrm{Mat}_{4b}(\mathbb{Z}_p),
\end{equation}
where the blocks are the elements of $\mathrm{Mat}_{2b}(\mathbb{Z}_p)$ given by
\begin{equation} \label{eq:L}
L_b=\begin{pmatrix}
\begin{matrix}0 & \lambda_1 \\ - \lambda_1 & 0\end{matrix} & & 0 \\
 & \ddots & \\
0 & & \begin{matrix}0 & \lambda_b \\ - \lambda_b & 0
\end{matrix}
\end{pmatrix}
\end{equation}
\begin{equation} \label{eq:M}
M_b=\begin{pmatrix}
\begin{matrix}0 & \mu_1 \\ - \mu_1 & 0\end{matrix} & & 0 \\
 & \ddots & \\
0 & & \begin{matrix}0 & \mu_b \\ - \mu_b & 0
\end{matrix}
\end{pmatrix}
\end{equation}
\begin{equation} \label{eq:J}
J_b=\begin{pmatrix}
\begin{matrix}
0 & -1 \\ 1 & \; \; 0\end{matrix} & & 0 \\
 & \ddots & \\
0 & &
\begin{matrix}0 & -1 \\ 1 & \; \; 0
\end{matrix}
\end{pmatrix}
\end{equation}
\begin{proposition} \label{prop:omega-Kodaira-type}
If $\lambda_j \mu_j \neq 1$ for all $j \in \{1, \ldots, b\}$, then the alternating form $\omega \in \mathrm{Alt}^2(V)$ defined in $\eqref{eq:symplectic-Kodaira}$ is a symplectic form  of Heisenberg type.
\end{proposition}
\begin{proof}
Standard Gaussian elimination shows that
\begin{equation} \label{eq:determinant-Omega}
\det \Omega_b = (1-\lambda_1 \mu_1)^2 (1- \lambda_2 \mu_2)^2 \cdots (1-\lambda_b \mu_b)^2,
\end{equation}
so $\omega$ is non-degenerate if and only if $1-\lambda_j \mu_j \neq 0$ for all $j \in \{1, \ldots, b \}$.

It remains to show that $\omega$ is of Heisenberg type. Under the natural duality $\mathrm{Alt}^2(V) \simeq \wedge ^2 V^{\vee}$, the alternating form $\omega$ corresponds to the $2$-form on $V^{\vee}$ written in coordinates as
\begin{equation} \label{eq:2-form-omega}
\begin{split}
& \sum_{j=1}^b  \lambda_j \, (\alpha_j \otimes 1) \wedge (\beta_j \otimes 1) + \sum_{j=1}^b \mu_j \, (1\otimes \alpha_j) \wedge (1\otimes \beta_j) \\
- & \sum_{j=1}^b  (\alpha_j \otimes 1) \wedge (1 \otimes \beta_j) - \sum_{j=1}^b (1\otimes \alpha_j) \wedge (\beta_j \otimes 1),
\end{split}
\end{equation}
whose image via $\xi \colon \wedge^2 V^{\vee} \to H^2(\Sigma_b \times \Sigma_b, \, \mathbb{Z}_p)$ can be obtained by replacing  the wedge product in \eqref{eq:2-form-omega} with the cup product and using \eqref{eq:cup-in-tensor}. So, making also use of \eqref{eq:lambda-mu}, we get
\begin{equation} \label{eq:xi(omega)}
\begin{split}
\xi(\omega) & =  \sum_{j=1}^b  \lambda_j \, \xi(\alpha_j \otimes 1, \, \beta_j \otimes 1) +  \sum_{j=1}^b \mu_j \, \xi(1 \otimes \alpha_j, \, 1 \otimes \beta_j)  \\ & -  \sum_{j=1}^b \xi(\alpha_j \otimes 1, \,  1 \otimes \beta_j)-  \sum_{j=1}^b \xi(1 \otimes \alpha_j,  \beta_j \otimes 1)  \\
& = \left( \sum_{j=1}^b \lambda_j \right) \gamma \otimes 1 + \left( \sum_{j=1}^b \mu_j \right) 1 \otimes \gamma - \sum_{j=1}^b \alpha_j \otimes \beta_j +\sum_{j=1}^b \beta_j \otimes \alpha_j \\
& =  \gamma \otimes 1 +  1 \otimes \gamma +\sum_{j=1}^b (\beta_j \otimes \alpha_j - \alpha_j \otimes \beta_j) = \delta,
\end{split}
\end{equation}
cf. Remark \ref{rem:class-of-diagonal}. This concludes the proof.
\end{proof}
We can now show the existence of symplectic forms of Heisenberg type.

\begin{proposition}  \label{prop:heis-forma-exist}
Let $b \geq 2$ and set $V=H_1(\Sigma_b \times \Sigma_b - \Delta, \, \mathbb{Z}_p)$. If $p \geq 5$, then $\mathrm{Alt}^2(V)$ contains symplectic forms of Heisenberg type. 
\end{proposition}
\begin{proof}
By Proposition \ref{prop:omega-Kodaira-type}, it suffices to find non-zero elements $\lambda_1, \ldots, \lambda_b$, $\mu_1, \ldots, \mu_b$ in $\mathbb{Z}_p$ such that 
\begin{equation}
\sum_{j=1}^b \lambda_j = \sum_{j=1}^b \mu_j =1, \quad \lambda_j \mu_j \neq 1 \; \; \textrm{for all } j \in \{1, \ldots, b\}.  
\end{equation}
Choose arbitrarily $\lambda_j$, with $j \in\{1,\ldots,b-1\}$, and $\mu_j$,  with $j \in\{1,\ldots,b-2\}$, such that $\lambda_j\mu_j\ne 1$ for all $j \in\{1,\ldots,b-2\}$. Then $\lambda_b$ is uniquely determined by $\lambda_b=1-\sum_{j=1}^{b-1}\lambda_j$, whereas  $\mu_{b-1}$ and $\mu_b$ are subject to the following conditions:
\begin{itemize}
\item $\mu_{b-1}+\mu_b$ is equal to a constant $c=1-\sum_{j=1}^{b-2}\mu_j$ 
\item $\mu_{b-1}\ne\lambda_{b-1}^{-1}$, $\mu_{b}\ne\lambda_{b}^{-1}$.
\end{itemize}
These requirements are in turn equivalent to $\mu_{b-1}\notin \{\lambda_{b-1}^{-1},\,c-\lambda_b^{-1}\}$. Now, if $p \geq 5$ this can be clearly satisfied, because there are more than two non-zero elements in $\mathbb Z_p$.
\end{proof}

\begin{remark} \label{rmk:3-does-not-work}
If $p=3$, the condition $\lambda_j\mu_j\ne 1$ implies  $\lambda_j=-\mu_j$ and so, if this holds for all $j$, we get 
\begin{equation}
1=\sum_{j=1}^{b}\lambda_j=-\sum_{j=1}^{b}\mu_j=-1, 
\end{equation}
a contradiction. Hence the alternating form $\omega$ defined in \eqref{eq:symplectic-Kodaira} is never symplectic when $p=3$. In fact, the existence of symplectic forms of Heisenberg type in the case $p=3$ is an interesting problem, but we will not develop this point here.   
\end{remark}

\begin{remark}
The notion of alternating form of Heisenberg type (namely, satisfying one of the equivalent conditions of Definition \ref{def:heisenberg-type}) and the study of the corresponding lifting  $\varphi_{\omega} \colon \mathsf{P}_2(\Sigma_b) \to \mathsf{Heis}(V,\, \omega)$ make also sense when $\omega$ is non-symplectic, i.e. degenerate. In this case, by Remark \ref{rmk: degenerate 2-form}, we obtain an epimorphism $\varphi \colon \mathsf{P}_2(\Sigma_b) \to \mathsf{Heis}(W, \, \omega)$, where $W$ is the quotient of $V$ by $V_0=\ker \omega$. This fact will be exploited in Subsection \ref{subsec:degenerate-Heis}, in order to construct what we will call the \emph{degenerate Heisenberg covers} of $\Sigma_b \times \Sigma_b$; in that situation, the primes $p=2, \, 3$ will be also allowed, but we will have to make the further request that $p$ divides $b+1$.  
\end{remark}

We end this section by explicitly describing a lifting $\varphi_{\omega} \colon \mathsf{P}_2(\Sigma_b) \longrightarrow \mathsf{Heis}(V, \, \omega)$ of $\phi \colon \mathsf{P}_2(\Sigma_b) \longrightarrow  V$, in the case where $\omega$ is one of the symplectic forms considered  in Proposition \ref{prop:omega-Kodaira-type}. First of all, $\phi$ can be explicitly described as the group homomorphism defined by
\begin{equation} \label{eq:phi}
\phi(\rho_{1j}) = r_{1j}, \quad \phi(\tau_{1j})= t_{1j}, \quad \phi(\rho_{2j}) = r_{2j}, \quad  \phi(\tau_{2j}) = r_{2j}, \quad \phi(A_{12})=0,
\end{equation}
for all $j \in \{1, \ldots, g \}$. On the other hand, the $4b+1$ elements of $V \times \mathbb{Z}_p$ given by
\begin{equation}
\mathsf{r}_{1j}=(r_{1j}, \, 0),\quad \mathsf{t}_{1j}=(t_{1j}, \, 0), \quad  \mathsf{r}_{2j}=(r_{2j}, \, 0), \quad  \mathsf{t}_{2j}=(t_{2j}, \, 0), \quad \mathsf{z}=(0, \, 1)
\end{equation}
are generators for $\mathsf{Heis}(V, \, \omega)$, subject to the same relations as in \eqref{eq:rel-heis}, cf. the proof of Proposition \ref{prop:Heisenberg-is-extra-special}. Thus, using  \eqref{eq:symplectic-Kodaira}, for all $j, \, k \in \{1, \ldots, b \}$ we obtain
\begin{equation} \label{eq:rel-heis-g}
\begin{split}
\mathsf{r}_{1j}^p & = \mathsf{t}_{1j}^p= \mathsf{r}_{2j}^p = \mathsf{t}_{2j}^p= \mathsf{z}^p=1 \\
[\mathsf{r}_{1j}, \, \z] & = [\mathsf{t}_{1j}, \, \z]= 
[\mathsf{r}_{2j}, \, \z] = [\mathsf{t}_{2j}, \, \z]=1 \\
[\mathsf{r}_{1j}, \, \mathsf{r}_{1k}]& =[\mathsf{t}_{1j}, \, \mathsf{t}_{1k}]=1 \\
[\mathsf{r}_{1j}, \, \mathsf{r}_{2k}]& =[\mathsf{t}_{1j}, \, \mathsf{t}_{2k}]=1 \\
[\mathsf{r}_{2j}, \, \mathsf{r}_{2k}]& =[\mathsf{t}_{2j}, \, \mathsf{t}_{2k}]=1 \\
[\mathsf{r}_{1j}, \,\mathsf{t}_{1 k}]& =\mathsf{z}^{\delta_{jk} \, \lambda_j} \\
[\mathsf{r}_{2j}, \,\mathsf{t}_{2 k}]& =\mathsf{z}^{\delta_{jk} \, \mu_j} \\
[\mathsf{r}_{1j}, \,\mathsf{t}_{2 k}]& = [\mathsf{r}_{2j}, \,\mathsf{t}_{1 k}]=\mathsf{z}^{-\delta_{jk}},
\end{split}
\end{equation}
where $\delta_{jk}$ is the usual Kronecker symbol.

\medskip 
Let us now  define $\varphi_{\omega}$ as
\begin{equation} \label{eq:varphi}
\varphi_{\omega}(\rho_{1j}) = \mathsf{r}_{1j}, \quad \varphi_{\omega}(\tau_{1j})= \mathsf{t}_{1j}, \quad \varphi_{\omega}(\rho_{2j}) = \mathsf{r}_{2j}, \quad  \varphi_{\omega}(\tau_{2j}) = \mathsf{r}_{2j}, \quad \varphi_{\omega}(A_{12})=\mathsf{z}.
\end{equation}
Looking at Theorem \ref{thm:presentation-braid},  it is straightforward to check that $\varphi_{\omega}$ defines a group homomorphism of $\mathsf{P}_2(\Sigma_b)$ onto $\mathsf{Heis}(V, \, \omega)$, providing a lifting of $\phi$ and such that  $\varphi_{\omega}(A_{12})$ has order $p$.

\begin{remark} \label{rmk:lifting}
In the above description of the lifting, the role of conditions \eqref{eq:lambda-mu} is to ensure that the map $\varphi_{\omega}$  is compatible with the two surface relations in the presentation of $\mathsf{P}_2(\Sigma_b)$. For instance, it is straightforward to check that the identities
\begin{equation}
[\mathsf{r}_{1j}^{-1}, \,\mathsf{t}_{1 j}^{-1}]=\mathsf{z}^{\lambda_j}, \quad j\in \{1, \ldots, b\}
\end{equation}
are consequence of  \eqref{eq:rel-heis-g}; thus, considering the image of the first surface relation via $\varphi_{\omega}$, we obtain
\begin{equation}
\begin{split}
& [\mathsf{r}_{1g}^{-1}, \, \mathsf{t}_{1g}^{-1}] \, \mathsf{t}_{1g}^{-1} \, [\mathsf{r}_{1 \,g-1}^{-1}, \, \mathsf{t}_{1 \,g-1}^{-1}] \, \mathsf{t}_{1\,g-1}^{-1} \cdots [\mathsf{r}_{11}^{-1}, \, \mathsf{t}_{11}^{-1}] \, \mathsf{t}_{11}^{-1} \, (\mathsf{t}_{11} \, \mathsf{t}_{12} \cdots \mathsf{t}_{1g}) \\
= & \mathsf{z}^{\lambda_g} \,  \mathsf{t}_{1g}^{-1} \, \mathsf{z}^{\lambda_{g-1}} \, \mathsf{t}_{1\,g-1}^{-1} \cdots  \mathsf{z}^{\lambda_1} \,  \mathsf{t}_{11}^{-1} \, (\mathsf{t}_{11} \, \mathsf{t}_{12} \cdots \mathsf{t}_{1g}) \\
= & \mathsf{z}^{\lambda_1 + \ldots +\lambda_g} \, (\mathsf{t}_{11} \, \mathsf{t}_{12} \cdots \mathsf{t}_{1g})^{-1} (\mathsf{t}_{11} \, \mathsf{t}_{12} \cdots \mathsf{t}_{1g}) =\mathsf{z},
\end{split}
\end{equation}
that is, the relation is invariant under $\varphi_{\omega}$. The invariance of the second surface relation can be proved analogously.
\end{remark}

\section{Application: construction of double Kodaira fibrations} \label{sec:Kodaira}

\subsection{Double Kodaira fibrations} \label{subsec:Kodaira-fib}

\begin{definition}
A \emph{Kodaira fibration} is a smooth holomorphic fibration $($with connected fibres$)$  $f_1 \colon S \to B_1$, where $S$ is a compact complex surface and $B_1$ is a compact complex curve, which is not isotrivial (this means that not all its fibres are biholomorphic to each others).
\end{definition}
Equivalently, by \cite{FiGr65}, a Kodaira fibration is a smooth connected fibration $f_1 \colon S \to B_1$ which is not a locally trivial holomorphic fibre bundle, cf.\cite[Chapter I, (10.1)]{BHPV03}; however, any such a fibration is a locally trivial differentiable fibre bundle in the category of real $C^{\infty}$ manifolds, because of the Ehresmann theorem \cite{Eh51}. The genus $b_1:=g(B_1)$ is called the \emph{base genus} of the fibration, whereas the genus $g:=g(F)$, where $F$ is any fibre, is called the \emph{fibre genus}. If a surface $S$ is the total space of a Kodaira fibration, we will call it a \emph{Kodaira fibred surface}.

By \cite[Theorem 1.1]{Kas68}, every Kodaira fibration $f_1 \colon S \to B_1$ satisfies $b_1 \geq 2$ and $g \geq 3$. In particular, $S$ contains no rational or elliptic curves: in fact, such curves can neither dominate the base (because $b_1 \geq 2$) nor be contained in fibres (since the fibration is smooth). So every Kodaira fibred surface $S$ is minimal and, by the superadditivity of the Kodaira dimension, it is of general type, hence algebraic.

Important invariants of a Kodaira fibred surface $S$ are the signature $\sigma(S)$ and the slope $\nu(S)$, given by
\begin{equation} \label{eq:slope+signature}
\sigma(S)=\frac{1}{3} \left( c_1^2(S)-2 c_2(S) \right), \quad \nu(S)=\frac{c_1^2(S)}{c_2(S)}.
\end{equation}
As explained in the Introduction (to which we refer the reader for a more detailed discussion on this topic), Kodaira fibrations were originally introduced in \cite{Kod67} and \cite{At69} in order to show that the signature $\sigma$ of a manifold, meaning the signature of the intersection form in the middle cohomology, is not multiplicative for  fibre bundles. In fact, every Kodaira fibred surface $S$ satisfies $\sigma(S) >0$, whereas the signature of the product of two curves is always zero.

It is rather difficult to construct Kodaira fibred surfaces with small $\sigma(S)$. Since $S$ is a differentiable $4$-manifold which is a real surface bundle, its signature is divisible by $4$, see \cite{Mey73}. If moreover $S$ has a spin structure, i.e., its canonical class is $2$-divisible in $\mathrm{Pic}(S)$, then by Rokhlin's theorem its signature is necessarily a positive multiple of $16$, and examples with $\sigma(S)=16$ are constructed in \cite{LLR17}. It is not known if there exists a Kodaira fibred surface such that $\sigma(S) \leq 12$.

The slope can be seen as a quantitative measure of the non-multiplicativity of the signature. In fact, every product surface $F \times B$ satisfies $\nu(F \times B)=2$; on the other hand, if $S$ is a Kodaira fibred surface, then Arakelov inequality (see \cite{Be82}) implies $\nu(S)>2$, while Liu inequality (see \cite{Liu96}) yields $\nu(S)<3$, so that for such a surface the slope lies in the open interval $(2, \, 3)$. The original examples by Atiyah, Hirzebruch and Kodaira have slope lying in $(2, 2 + 1/3]$, see \cite[p. 221]{BHPV03}, and the first examples with higher slope appeared in \cite{CatRol09}, where it is shown that there are Kodaira surfaces with slope equal to $2+2/3$. This is the record for the slope so far, in particular it is a present unknown whether the slope of a Kodaira fibred surface can be arbitrarily close to $3$. 

The examples provided in \cite{CatRol09} are actually rather special cases of Kodaira fibred surfaces, called \emph{double Kodaira surfaces}. 

\begin{definition}
A \emph{double Kodaira surface} is a compact complex surface $S$, endowed with a \emph{double Kodaira fibration}, namely a surjective, holomorphic map $f \colon S \to B_1 \times B_2$ yielding, by composition with the natural projections, two Kodaira fibrations $f_i \colon S \to B_i$, $i=1, \,2$.
\end{definition}

Note that a surface $S$ is a double Kodaira surface if and only if it admits two distinct Kodaira fibrations $f_i \colon S \to B_i$, $i=1, \,2$, since in this case we can take as $f$ the product morphism $f_1 \times f_2$.

\begin{definition} \label{def:simple-very-symple-standard}
Let $f \colon S \to B_1 \times B_2$ be a double Kodaira fibration, and let $D \subset B_1 \times B_2$ be the branch locus of $f$. Then $f$ is called
\begin{itemize}
\item \emph{double {\'e}tale} if $D$ is smooth and the projections of $D$ over $B_1$ and $B_2$  are both {\'e}tale maps;
\item \emph{simple} if it is double {\'e}tale and there exist {\'e}tale maps $\phi_1, \ldots \phi_k$ from $B_1$ to $B_2$ such that $D$ is the disjoint union of their graphs $D_{\phi_1}, \ldots, D_{\phi_k}$;
\item \emph{very simple} if it is simple and moreover $B_1=B_2$ and $\phi_1, \ldots, \phi_k$ are automorphisms;
\item \emph{standard} if there exists {\'e}tale Galois covers $B \to B_i$, $i=1, \, 2$ such that the {\'e}tale pullback
\begin{equation*}
\tilde{S}:= S \times_{(B_1 \times B_2)}(B \times B)
\end{equation*}
induced by $B \times B \to B_1 \times B_2$ is very simple.
\end{itemize}
\end{definition}
Up to now, there are no known example of double {\'e}tale Kodaira fibrations that are not standard.

\subsection{Double Kodaira fibrations via Galois covers of $\Sigma_b \times \Sigma_b$} \label{subsec:double-Kod-via-Galois}

Now our aim is to construct some double Kodaira fibrations by using the techniques developed in the previous sections. The key result in this direction will be the following version of Riemann Extension Theorem, that follows in a straightforward manner from \cite[Corollary 1]{Pol18}.
\begin{proposition} \label{prop:Riemann-ext}
Let $Y$ be a smooth projective variety, $Z \subset Y$ a smooth irreducible divisor and $G$ a finite group. Then the isomorphism classes of connected Galois covers $\mathbf{f} \colon X \to Y$ with Galois group $G$, branched at most over $Z$, are in bijection to group epimorphisms $\varphi \colon \pi_1(Y - Z) \to G$, up to automorphisms of $G$.
\end{proposition}

As an important consequence, we obtain the following recipe to construct double Kodaira fibrations. With a slight abuse of notation, in the sequel we will use the symbol $\Sigma_b$ to indicate both a smooth complex projective curve of genus $b$ and its underlying real surface.

\begin{theorem} \label{thm:double-Kodaira-from-G}
Let $G$ be a finite group such that there is a group epimorphism $\varphi \colon \mathsf{P}_2(\Sigma_b) \to G,$ and denote by
\begin{equation}
\varphi_1 \colon \pi_1(\Sigma_b - \{p_1\}, \, p_2) \to G, \quad \varphi_2 \colon \pi_1(\Sigma_b - \{p_2\}, \, p_1) \to G
\end{equation}
the compositions of $\varphi$ with the monomorphisms $\pi_1(\Sigma_b - \{p_1\}, \, p_2) \to  \mathsf{P}_2(\Sigma_b)$ and $\pi_1(\Sigma_b - \{p_2\}, \, p_1) \to  \mathsf{P}_2(\Sigma_b)$, respectively $($cf. \eqref{eq:split-braid} and \eqref{eq:split-braid-new}$)$. Assume that
\begin{itemize}
\item the element $\varphi(A_{12})$ has order $n>1;$
\item the image of $\varphi_i$ has index $m_i$ in $G,$ for $i=1, \, 2$.
\end{itemize}
Then there is a double Kodaira fibration $f \colon S \to \Sigma_{b_1} \times \Sigma_{b_2},$ where
\begin{equation} \label{eq:expression-gi}
b_1 -1 =m_1(b-1), \quad  b_2 -1 =m_2(b-1).
\end{equation}
Moreover, the fibre genera $g_1$, $g_2$ of the Kodaira fibrations $f_1 \colon S \to \Sigma_{b_1}$, $f_2 \colon S \to \Sigma_{b_2}$ are computed by the formulae
\begin{equation} \label{eq:expression-gFi}
2g_1-2 = \frac{|G|}{m_1} (2b-2 + \mathfrak{n} ), \quad 2g_2-2 = \frac{|G|}{m_2} \left( 2b-2 + \mathfrak{n} \right),
\end{equation}
where $\mathfrak{n}:= 1 - 1/n$.
\end{theorem}
\begin{proof}
Set $Y= \Sigma_b \times \Sigma_b$ and $Z = \Delta$. Using Proposition \ref{prop:Riemann-ext} and the isomorphism \eqref{eq:iso-braids}, from the existence of the group epimorphism $\varphi \colon \mathsf{P}_2(\Sigma_b) \to G$ we obtain the existence of a Galois cover $\mathbf{f} \colon S \to \Sigma_b \times \Sigma_b$, whose Galois group is isomorphic to $G$. Furthermore, since $\varphi(A_{12})$ has order $n$, part ($v$) of Remark \ref{rmk:presentation-braid} implies that such a cover is branched exactly over $\Delta$, with branching order $n$.  The diagonal $\Delta$ intersects transversally at a single point the fibres of both natural projections $\pi_i \colon \Sigma_b \times \Sigma_b \to \Sigma_b$, and such a point is different for each fibre.  So, taking the Stein factorizations of the compositions $\pi_i \circ \mathbf{f} \colon S \to \Sigma_b$ as in the diagram below
\begin{equation} \label{dia:Stein-Kodaira-gi}
\begin{tikzcd}
  S \ar{rr}{\pi_i \circ \mathbf{f}}  \ar{dr}{f_i} & & \Sigma_b  \\
   & \Sigma_{b_i} \ar{ur} &
  \end{tikzcd}
\end{equation}
we obtain two distinct Kodaira fibrations $f_i \colon S \to \Sigma_{b_i}$, hence a double Kodaira fibration by considering the product morphism  \begin{equation}
f=f_1 \times f_2 \colon S \to \Sigma_{b_1} \times \Sigma_{b_2}.
\end{equation}
If $\Gamma$ is the fibre of $\pi_1 \circ \mathbf{f} \colon S \to \Sigma_b$ over the point $p_1 \in \Sigma_b$, we see that the restriction  $\mathbf{f}|_{\Gamma} \colon \Gamma \to \Sigma_b$ is the smooth Galois cover induced by the group homomorphism $\varphi_1 \colon \pi_1(\Sigma_b - \{p_1\}, \, p_2) \to G$, whose image, by assumption, has index $m_1$ in $G$. It follows that the Galois cover $\mathbf{f}|_{\Gamma} \colon \Gamma \to \Sigma_b$ splits into the disjoint union of $m_1$ disjoint covers of the form $\Gamma' \to \Sigma_b$; each of them is a connected Galois cover, with Galois group isomorphic to $\mathrm{im}(\varphi_1)$, branched precisely at one point with branching order $n$. But then this is true for every fibre of $\pi_1 \circ \mathbf{f}$, because the number of connected components of the fibres of a proper, normal, flat morphism is constant, see \cite[Theorem 4.17 ($iii$)]{DM69}.

Summing up, the Stein factorization in diagram \eqref{dia:Stein-Kodaira-gi} is obtained by taking an \'{e}tale cover $\theta_1 \colon \Sigma_{b_1} \to \Sigma_b$ of degree $m_1$, and every fibre of the Kodaira fibration $f_1 \colon S \to \Sigma_{b_1}$ is homeomorphic to $\Gamma'$; thus we can compute both $b_1$ and $g_1$ by using the Hurwitz formula, obtaining the first relations in \eqref{eq:expression-gi} and \eqref{eq:expression-gFi}. The computation of  $b_2$ and $g_2$ is exactly the same, so we are done.
\end{proof}

\begin{remark} \label{rmk:non-abelian-G}
In Theorem \ref{thm:double-Kodaira-from-G}, the group $G$ is necessarily \emph{non-abelian}. Otherwise, since $A_{12}$ is a commutator in $\mathsf{P}_2(\Sigma_b)$, the element $\varphi(A_{12})$ would be trivial.
\end{remark}

\begin{remark} \label{rmk:degree-f-general}
Looking at the proof of Theorem \ref{thm:double-Kodaira-from-G}, we see that there is a commutative diagram
\begin{equation} \label{dia:degree-f-general}
\begin{tikzcd}
  S \ar{rr}{\mathbf{f}}  \ar{dr}{f} & & \Sigma_b \times \Sigma_b  \\
   & \Sigma_{b_1} \times \Sigma_{b_2} \ar[ur, "\theta_1 \times \theta_2"{sloped, anchor=south}] &
  \end{tikzcd}
\end{equation}
Since $\theta_1 \times \theta_2 \colon \Sigma_{b_1} \times \Sigma_{b_2} \to \Sigma_b \times \Sigma_b$ is \'{e}tale of degree $m_1 m_2$, it follows that the double Kodaira fibration $f \colon S \to  \Sigma_{b_1} \times \Sigma_{b_2}$ has degree $\frac{|G|}{m_1m_2}$ and is branched precisely over the curve
\begin{equation} \label{eq:branching-f}
(\theta_1 \times \theta_2)^{-1}(\Delta)=\Sigma_{b_1} \times_{\Sigma_b} \Sigma_{b_2}.
\end{equation}
Such a curve is always smooth, being the preimage of a smooth divisor via an \'{e}tale morphism. However, it can be reducible, cf. Proposition \ref{prop:degree-double-Kodaira}.
\end{remark}

We can now compute the invariants of $S$, cf. \cite[p. 55, formulae (3.9) and (3.10)]{Tr16}.

\begin{proposition} \label{prop:invariant-S-G}
Let $f \colon S \to \Sigma_{b_1} \times \Sigma_{b_2}$ be a double Kodaira fibration as in Theorem \emph{\ref{thm:double-Kodaira-from-G}}. Then the invariants of $S$ are
\begin{equation} \label{eq:invariants-S-G}
\begin{split}
c_1^2(S) & = |G|\,(2b-2) ( 4b-4 + 4 \mathfrak{n} - \mathfrak{n}^2 ) \\
c_2(S) & =   |G|\,(2b-2) (2b-2 + \mathfrak{n}),
\end{split}
\end{equation}
and so the slope and the signature of $S$ can be expressed as
\begin{equation} \label{eq:slope-signature-S-G}
\begin{split}
\nu(S) & = \frac{c_1^2(S)}{c_2(S)} = 2+ \frac{2 \mathfrak{n} - \mathfrak{n}^2}{2b-2 + \mathfrak{n} } \\
\sigma(S) & = \frac{1}{3}\left(c_1^2(S) - 2 c_2(S) \right) = \frac{1}{3} \, |G| \, (2b-2)(2 \mathfrak{n} - \mathfrak{n}^2).
\end{split}
\end{equation}
\begin{proof}
Let $F_1$ be any fibre of the the Kodaira fibration $f_1 \colon S \to \Sigma_{b_1}$. Since $f_1$ is a smooth holomorphic fibration, it follows by \cite[Lemma VI.4]{Be96} that the topological Euler number of $S$ is given by
\begin{equation*}
c_2(S)=\chi_{\mathrm{top}}(\Sigma_{b_1}) \chi_{\mathrm{top}}(F_1) =  (2-2b_1)(2-2g_1),
\end{equation*}
and the last expression equals $|G|\,(2b-2) (2b-2 + \mathfrak{n})$ by \eqref{eq:expression-gi} and \eqref{eq:expression-gFi}.

It remains to compute $c_1^2(S)$. To this pourpose, note that the ramification divisor of the Galois cover $\mathbf{f} \colon S \to \Sigma_b \times \Sigma_b$ equals $(n-1)D$, where $D \subset S$ is the effective divisor such that $\mathbf{f}^* \Delta = n D$. Then by the Hurwitz formula we deduce
\begin{equation} \label{eq:Hurwitz-S}
K_S = \mathbf{f}^* K_{\Sigma_b \times \Sigma_b} + (n-1)D,
\end{equation}
and so
\begin{equation} \label{eq:computation-c1^2}
\begin{split}
c_1^2(S)  & = (\mathbf{f}^* K_{\Sigma_b \times \Sigma_b})^2 + 2 \mathbf{f}^* K_{\Sigma_b \times \Sigma_b} \cdot (n-1) D + (n-1)^2 D^2 \\
& = |G| \, K_{\Sigma_b \times \Sigma_b}^2 + 2 (n-1)  K_{\Sigma_b \times \Sigma_b} \cdot \mathbf{f}_*D + (n-1)^2 \left(\frac{1}{n} \mathbf{f}^* \Delta \right)^2 \\
& = 8 \, |G| \, (b-1)^2 + 2 (n-1) K_{\Sigma_b \times \Sigma_b} \cdot \frac{|G|}{n} \Delta + \left( \frac{n-1}{n} \right)^2 |G| \, \Delta^2
\\
& = 8 \, |G| \, (b-1)^2 + 2 \left(\frac{n-1}{n} \right) |G| (4b-4) + \left( \frac{n-1}{n} \right)^2 |G|\, (2-2b) \\
& = \, |G| \, (2b-2) ( 4b-4 + 4 \mathfrak{n} - \mathfrak{n}^2 ).
\end{split}
\end{equation}
This completes the proof.
\end{proof}
\end{proposition}

\begin{remark} \label{rmk:Rokhlin}
If  either $n$ is odd or $D$ is $2$-divisible in $\mathrm{Pic}(S)$, then \eqref{eq:Hurwitz-S} shows that $K_S$ is $2$-divisible in $\mathrm{Pic}(S)$. Hence $S$ is a compact, oriented smooth spin $4$-manifold and so, by Rokhlin's Theorem \cite[p. 60]{Mo96}, its signature is divisible by $16$.
\end{remark}

\subsection{Non-degenerate Heisenberg covers of $\Sigma_b \times \Sigma_b$} \label{subsec:double-Kod-from-Heis}

We are now ready to apply the results of Subsection \ref{subsec:double-Kod-via-Galois} to our situation, where the role of $G$ will be played by the group $\mathsf{Heis}(V, \, \omega)$. In this way, we will obtain some Galois covers of $\Sigma_b \times \Sigma_b$, with Galois group $\mathsf{Heis}(V, \, \omega)$ and branched over the diagonal $\Delta$, that we will call \emph{$($non-degenerate$)$ Heisenberg covers}. As explained in the proof of Theorem \ref{thm:double-Kodaira-from-G}, the Stein factorizations of these covers provide  the double Kodaira fibrations we are looking for.

We use the same notation as in Section \ref{sec:rep-braid}, in particular we denote by $p$ an odd prime and by $V$ the $\mathbb{Z}_p$-vector space $H_1(\Sigma_b \times \Sigma_b - \Delta, \, \mathbb{Z}_p)$, where $b \geq 2$.
\begin{theorem} \label{thm:double-Kodaira-from-omega}
Assume $p \geq 5$ and let $\omega \in \wedge^2 V^{\vee}$ be a symplectic form of Heisenberg type, that exists by Proposition \eqref{prop:heis-forma-exist}. Then the group epimorphism $\varphi_{\omega} \colon \mathsf{P}_2(\Sigma_b) \to \mathsf{Heis}(V, \, \omega)$ defined in \eqref{eq:varphi} induces a Heisenberg cover $\mathbf{f} \colon S_{b, \, p} \to \Sigma_b \times \Sigma_b$ that, after applying the Stein factorization, gives rise to a double Kodaira fibration $f \colon S_{b, \, p} \to \Sigma_{b'} \times \Sigma_{b'}$, where
\begin{equation} \label{eq:g'}
b'-1=p^{2b}(b-1).
\end{equation}
Moreover, the two Kodaira fibrations $f_i \colon S_{b, \, p} \to \Sigma_{b'}$ have the same fibre genus $g$, which is related to $b$ by the formula
\begin{equation} \label{eq:g-gprime}
2g-2 = p^{2b+1} \left(2b-2 + \mathfrak{p}\right),
\end{equation}
where $\mathfrak{p}:= 1 - 1/p$. Finally, the invariants of $S_{b, \, p}$ are
\begin{equation} \label{eq:invariants-S}
\begin{split}
c_1^2(S_{b, \, p}) & = p^{4b+1}(2b-2) (4b-4 + 4 \mathfrak{p} - \mathfrak{p}^2) \\
c_2(S_{b, \, p}) & =  p^{4b+1}(2b-2)(2b-2 + \mathfrak{p}),
\end{split}
\end{equation}
and so the slope and the signature of $S_{b, \, p}$ can be expressed as
\begin{equation} \label{eq:slope-signature-S}
\begin{split}
\nu(S_{b, \, p}) & = \frac{c_1^2(S_{b, \, p})}{c_2(S_{b, \, p})} = 2+ \frac{2 \mathfrak{p} - \mathfrak{p}^2}{2b-2 + \mathfrak{p} } \\
\sigma(S_{b, \, p}) & = \frac{1}{3}\left(c_1^2(S_{b, \, p}) - 2 c_2(S_{b, \, p}) \right) = \frac{1}{3} p^{4b+1}(2b-2)(2 \mathfrak{p} - \mathfrak{p}^2).
\end{split}
\end{equation}
\end{theorem}
\begin{proof}
First of all, let us recall that $\mathsf{z}=\varphi_{\omega}(A_{12})$ has order $p$ in $\mathsf{Heis}(V, \, \omega)$. Let us consider now the group homomorphism
\begin{equation} \label{eq:composed-homomorphism}
\varphi_1 \colon \pi_1(\Sigma_b - \{p_1\}, \, p_2) \to \mathsf{Heis}(V, \, \omega),
\end{equation}
given by the composition of $\varphi_{\omega}$ with  the monomorphism $\pi_1(\Sigma_b - \{p_1\}, \, p_2) \to  \mathsf{P}_2(\Sigma_b)$. Its image is the subgroup of $\mathsf{Heis}(V, \, \omega)$ generated by
\begin{equation*}
\mathsf{r}_{21}, \ldots, \mathsf{r}_{2b}, \; \mathsf{t}_{21}, \ldots, \mathsf{t}_{2b}, \; \mathsf{z},
\end{equation*}
and this is in turn isomorphic to $\mathsf{Heis}(V_1, \, \omega_1)$, where $V_1 = H_1(\Sigma_b - \{p_1 \}, \, \mathbb{Z}_p)$ is generated by $r_{21}, \ldots {r}_{2b}$, $t_{21}, \ldots, t_{2b}$ and $\omega_1$ is the restriction of $\omega$ to $V_1$, namely the symplectic form $\omega_1 \colon V_1 \times V_1 \to \mathbb{Z}_p$ defined by the matrix $M_b$, see \eqref{eq:M}; note that $\mathsf{Heis}(V_1, \, \omega_1)$ has order $p^{2b+1}$. Clearly, the same argument works after exchanging the roles of $p_1$ and $p_2$; thus, both $\mathrm{im}(\varphi_1)$ and $\mathrm{im}(\varphi_2)$ have index $p^{2b}$ in $\mathsf{Heis}(V, \, \omega)$, and the result follows by applying Theorem \ref{thm:double-Kodaira-from-G} and Proposition \ref{prop:invariant-S-G}.
\end{proof}

Let us now describe in more detail the structure of the double Kodaira fibration $f \colon S_{b, \, p} \to \Sigma_{b'} \times   \Sigma_{b'}$; it will turn out that it is very simple, in the sense of Definition \ref{def:simple-very-symple-standard}.

\begin{proposition} \label{prop:degree-double-Kodaira}
The finite morphism $f \colon S_{b, \, p} \to \Sigma_{b'} \times   \Sigma_{b'}$ is a cyclic cover of degree $p$, totally branched on a reducible divisor
\begin{equation} \label{eq:degree-double-Kodaira}
D = \sum_{c \in (\mathbb{Z}_p)^{2b}} D_{c},
\end{equation}
where the $D_{c}$ are pairwise disjoint graphs of automorphisms of $\Sigma_{b'}$. Moreover, $D_0$ coincides with the diagonal $\Delta' \subset \Sigma_{b'} \times   \Sigma_{b'}$.
\end{proposition}
\begin{proof}
Let us first observe that the two subgroups
\begin{equation}
\mathrm{im}(\varphi_1) = \langle \mathsf{r}_{21}, \ldots, \mathsf{r}_{2b}, \; \mathsf{t}_{21}, \ldots, \mathsf{t}_{2b}, \; \mathsf{z} \rangle, \quad \mathrm{im}(\varphi_2) = \langle \mathsf{r}_{11}, \ldots, \mathsf{r}_{1b}, \; \mathsf{t}_{11}, \ldots, \mathsf{t}_{1b}, \; \mathsf{z} \rangle
\end{equation}
of $\mathsf{Heis}(V, \, \omega)$ are both normal (because they contain the commutator subgroup $\langle z \rangle$) and are exchanged by an automorphism of order $2$. Thus the two \'{e}tale covers
\begin{equation}
\theta_1 \colon \Sigma_{b'} \to \Sigma_b, \quad \theta_2 \colon \Sigma_{b'} \to \Sigma_b
\end{equation}
are both Galois, with Galois group $(\mathbb{Z}_p)^{2b}$ and, up to an involution of $\Sigma_{b'}$, we may assume $\theta_1 = \theta_2 = \theta$. This implies that the cover
\begin{equation}
\theta \times \theta \colon \Sigma_{b'} \times \Sigma_{b'} \to \Sigma_b \times \Sigma_b
\end{equation}
is Galois, with Galois group $(\mathbb{Z}_p)^{2b} \times (\mathbb{Z}_p)^{2b}$, and so
$f \colon S_{b, \, p} \to \Sigma_{b'} \times \Sigma_{b'}$ is also Galois, with Galois group isomorphic to the kernel of the group epimorphism
\begin{equation}
\mathsf{Heis}(V, \, \omega) \to (\mathbb{Z}_p)^{2b} \times (\mathbb{Z}_p)^{2b},
\end{equation}
see \cite[Lemma 4.1]{Pig19}. Since such a kernel has order $p$, it follows that $f$ is a cyclic cover of degree $p$, in particular it is totally branched because $p$ is prime. Remark \ref{rmk:degree-f-general} tells us that the branch locus of $f$ is smooth and it is given by the fibre product
\begin{equation}
\begin{split}
D = \Sigma_{b'} \times_{\Sigma_b} \Sigma_{b'} = & \left\{\, (x, \, y) \; \; | \; \; x, \, y \in \Sigma_{b'} \;\; \mathrm{and} \; \; \theta(x)=\theta(y) \,  \right\} \\
= & \left\{\, (x, \, c \cdot x) \; \; | \;\; x \in \Sigma_{b'}, \, c \in (\mathbb{Z}_p)^{2b} \, \right\},
\end{split}
\end{equation}
so we can end the proof by setting
\begin{equation}
D_{c} =\left\{\, (x, \, c \cdot x) \; \; | \;\; x \in \Sigma_{b'} \, \right\}.
\end{equation}
\end{proof}

\begin{proposition} \label{prop:invariants-nondegenerate-case}
Let $f \colon S_{b, \, p} \to \Sigma_{b'} \times   \Sigma_{b'}$ be a double Kodaira fibration as in Theorem $\mathrm{\ref{thm:double-Kodaira-from-omega}}$.
\begin{itemize}
\item[$\boldsymbol{(1)}$] We have
\begin{equation} \label{eq:slope-1}
2  < \nu(S_{b, \, p} ) \leq 2 + \frac{12}{35},
\end{equation}
and equality on the right holds precisely in two cases, namely
\begin{equation}
\nu(S_{2, \, 5}) = \nu(S_{2, \, 7})= 2 + \frac{12}{35}.
\end{equation}
Moreover, $\nu(S_{2, \, p}) > 2 + 1/3$ for all $p \geq 5$. More precisely, if $p \geq 7$ the function $\nu(S_{2, \, p})$ is strictly decreasing and
\begin{equation*}
\lim_{p \rightarrow +\infty} \nu(S_{2, \, p}) = 2 + \frac{1}{3}.
\end{equation*}

\item[$\boldsymbol{(2)}$] The signature $\sigma(S_{b, \, p})$ is divisible by $16$ and moreover
\begin{equation}
\sigma(S_{b, \, p}) \geq \sigma(S_{2, \,5})=2^4 \cdot 5^7.
\end{equation}
\end{itemize}
\end{proposition}
\begin{proof}
$\boldsymbol{(1)}$ From \eqref{eq:slope-signature-S} we get
\begin{equation} \label{eq:bounds-slope}
2 < \nu(S_{b, \, p}) = 2+ \frac{2 \mathfrak{p} - \mathfrak{p}^2}{2b-2 + \mathfrak{p} } =  2+ \frac{p^2-1}{(2b-1)p^2-p}.
\end{equation}
Since for $p \rightarrow + \infty$ the last expression is asymptotic to $2+ \frac{1}{2b-1}$, we see that $\nu(S_{b, \, p})$ is arbitrarily close to $2$ if both $b$ and $p$ are large enough. Moreover, we also have
\begin{equation}
\nu(S_{b, \, p}) \leq \nu(S_{2, \, p}) = 2 +  \frac{p^2-1}{3p^2-p},
\end{equation}
with equality holding if and only if $b=2$. The rest of the proof is a consequence of the following easy facts: for $x \geq 2$, the real function
\begin{equation*}
u(x)=\frac{x^2-1}{3x^2-x}
\end{equation*}
has a local (and global) maximum at $x_0 = 3+2 \sqrt{2} \simeq 5.828$, and for $x \geq x_0$ it is monotonically decreasing, with $\lim\limits_{x \rightarrow +\infty} u(x) = \frac{1}{3}$ from above. \\ \\ \noindent
$\boldsymbol{(2)}$ We can rewrite the second equality in \eqref{eq:slope-signature-S} as
\begin{equation} \label{eq:alternative-signature}
\sigma(S_{b, \, p})=\frac{1}{3}(2b-2)p^{4b-1}(p^2-1),
\end{equation}
and from this it is clear that $\sigma(S_{b, \, p})$ is a multiple of $16$ (since $p$ is odd, this also follows from Remark \ref{rmk:Rokhlin}) and that  the minimum value of $\sigma(S_{b, \, p})$ is attained precisely when $(b, \, p)=(2, \, 5)$.
\end{proof}

\begin{remark} \label{rmk:pairs}
Equations \eqref{eq:g'} and \eqref{eq:g-gprime} allow us to have complete control on both the base genus $b'$ and the fibre genus $g$ for the Kodaira fibrations $S_{b, \, p} \to \Sigma_{b'}$. For example, looking at the slope-maximizing cases mentioned in Proposition \ref{prop:invariants-nondegenerate-case} we obtain
\begin{equation} \label{eq:particular-cases}
\begin{split}
S_{2, \, 5}: \quad & (b', \, g)=(626, \, 4376) \\
S_{2, \, 7}: \quad & (b', \, g)=(2402, \, 24011). \\
\end{split}
\end{equation}
\end{remark}

\subsection{Degenerate Heisenberg covers of $\Sigma_b \times \Sigma_b$} \label{subsec:degenerate-Heis}
Let us consider now the situation where the alternating form $\omega \colon V \times V \to \mathbb{Z}_p$ satisfies one of the equivalent conditions in Definition \ref{def:Heis}, but it is not symplectic; let us denote as usual by  $V_0$ the non-trivial kernel of $\omega$, and by $W=V/V_0$ the corresponding quotient vector space.  Then, composing the group epimorphism $\mathsf{Heis}(V, \, \omega) \to \mathsf{Heis}(W, \, \omega)$ defined in Remark \ref{rmk: degenerate 2-form} with the homomorphism $\varphi_{\omega } \colon \mathsf{P}_2(\Sigma_b) \to \mathsf{Heis}(V, \, \omega)$ given by \eqref{eq:varphi}, we obtain a group epimorphism $\varphi \colon \mathsf{P}_2(\Sigma_b) \to \mathsf{Heis}(W, \, \omega)$ such that $\varphi(A_{12})=\mathsf{z}$. By Proposition \ref{prop:Riemann-ext}, this in turn yields a Galois cover $\mathbf{f} \colon S \to \Sigma_b \times \Sigma_b$, with Galois group isomorphic to $\mathsf{Heis}(W, \, \omega)$ and branched exactly over $\Delta$ with branching order $p$; such a $\mathbf{f}$ will be called a \emph{degenerate Heisenberg cover}. We will not attempt here to classify all possible covers of this type, and we will limit ourselves to analyze the following important example.

We assume that $p$ divides $b+1$, so that $-b=1$ holds in $\mathbb{Z}_p$, and we take
\begin{equation}
\lambda_1=\ldots=\lambda_b=\mu_1=\ldots \mu_b = -1 \in \mathbb{Z}_p.
\end{equation}
Therefore relations \eqref{eq:lambda-mu} are satisfied, so \eqref{eq:xi(omega)} shows that the corresponding alternating form $\omega$ satisfies $\xi(\omega)=\delta$. However, $\omega$ is not symplectic, since its associate matrix
\begin{equation} \label{eq:omega-degenerate}
\Omega_b=
\begin{pmatrix}
J_b & J_b \\
J_b & J_b
\end{pmatrix} \in \mathrm{Mat}_{4b}(\mathbb{Z}_p)
\end{equation}
has rank $2b$ and consequently $\omega$ has a $2b$-dimensional kernel $V_0$, namely
\begin{equation} \label{ker-omega}
V_0 = \langle r_{11}-r_{21}, \, t_{11}-t_{21}, \ldots, r_{1b}-r_{2b}, \, t_{1b}-t_{2b} \rangle.
\end{equation}
Then the images of $r_{11}, \, t_{11}, \ldots, r_{1b}, \, t_{1b}$ in the quotient vector space $W=V/V_0$ equal the images of $r_{21}, \, t_{21}, \ldots, r_{2b}, \, t_{2b}$, respectively. Let us denote such images  by $r_1, \, t_1, \ldots, r_b,\, t_b$ and let us call $\mathsf{r}_1, \, \mathsf{t}_1, \ldots, \mathsf{r}_b,\, \mathsf{t}_b$ the corresponding generators of the Heisenberg group $\mathsf{Heis}(W, \, \omega)$. Then, looking at \eqref{eq:rel-heis-g}, we see that a presentation for $\mathsf{Heis}(W, \, \omega)$ is as follows: for all $j,\, k \in \{1, \ldots, b \}$ we have
\begin{equation} \label{eq:rel-heis-g-degenerate}
\begin{split}
 \mathsf{r}_{j}^p & = \mathsf{t}_{j}^p=\mathsf{z}^p=1 \\
[\mathsf{r}_{j}, \, \mathsf{z}] &  = [\mathsf{t}_{j}, \, \mathsf{z}]= 1 \\
[\mathsf{r}_j, \mathsf{r}_k] & = [\mathsf{t}_j, \mathsf{t}_k] = 1 \\
[\mathsf{r}_{j}, \,\mathsf{t}_{k}] & =\mathsf{z}^{- \delta_{jk}}.
\end{split}
\end{equation}
Furthermore, in this situation, the group epimorphism $\varphi \colon \mathsf{P}_2(\Sigma_b) \to \mathsf{Heis}(W, \, \omega)$ is explicitly given by
\begin{equation} \label{eq:varphi-final}
\varphi(\rho_{1j}) = \varphi(\rho_{2j})=\mathsf{r}_j, \quad \varphi(\tau_{1j}) = \varphi(\tau_{2j})=\mathsf{t}_j, \quad \varphi(A_{12})= \mathsf{z}.
\end{equation}

In the situation of degenerate Heisenberg covers, we can deal with the case $p=2$, too. In this situation we cannot define the group $\mathsf{Heis}(W, \, \omega)$, see Remark \ref{rmk:two-descriptions-heisenberg}, but we can use the matrix group $\mathsf{H}_{2b+1}(\mathbb{Z}_2)$, which has order $2^{2b+1}$ and exponent $4$ (Remark \ref{rmk:heis-2-group}). A presentation for it is obtained by simply putting $p=2$ in \eqref{eq:rel-heis-g-degenerate}, i.e., for all $j, \, k  \in \{1, \ldots, b \}$ we have
\begin{equation} \label{eq:rel-heis-g-degenerate-2}
\begin{split}
 \mathsf{r}_{j}^2 & = \mathsf{t}_{j}^2=\mathsf{z}^2=1 \\
[\mathsf{r}_{j}, \, \mathsf{z}] &  = [\mathsf{t}_{j}, \, \mathsf{z}]= 1 \\
[\mathsf{r}_j, \mathsf{r}_k] & = [\mathsf{t}_j, \mathsf{t}_k] = 1 \\
[\mathsf{r}_{j}, \,\mathsf{t}_{k}] & =\mathsf{z}^{- \delta_{jk}}.
\end{split}
\end{equation}
So, if $2$ divides $b+1$ (i.e., if $b$ is odd) we can still consider a degenerate Heisenberg cover $\mathbf{f}  \colon S \to \Sigma_b \times \Sigma_b$, namely the one induced by the group epimorphism  $\varphi \colon \mathsf{P}_2(\Sigma_b) \to \mathsf{H}_{2b+1}(\mathbb{Z}_2)$ given exactly as in \eqref{eq:varphi-final}. Furthermore, since for an odd prime $p$ the group $\mathsf{Heis}(W, \omega)$ is isomorphic to $\mathsf{H}_{2b+1}(\mathbb{Z}_p)$, we can treat both the cases $p$ odd and $p=2$ at once, by considering \eqref{eq:varphi-final}
as a group epimorphism $\varphi \colon \mathsf{P}_2(\Sigma_b) \to \mathsf{H}_{2b+1}(\mathbb{Z}_p)$.

Interestingly, Theorem \ref{thm:directly-Kodaira} below will show that corresponding degenerate Heisenberg cover directly yields a double Kodaira fibration, that is, in this situation no Stein factorization is needed; in the sequel, we will use the ``$\circ$'' superscript in order to emphasize this fact, writing $f \colon S_{b, \, p}^{\circ} \to \Sigma_b \times \Sigma_b$ instead of $\mathbf{f}  \colon S_{b, \, p} \to \Sigma_b \times \Sigma_b$ when dealing with degenerate Heisenberg covers. 

This is in some respect an advantage with respect to the construction with a symplectic $\omega$, because, since we are not forced to take finite {\'e}tale covers, we get Kodaira surfaces with \emph{any} base genus $b$ and, moreover, we can choose as base curve $\Sigma_b$ \emph{any} curve of genus $b$. However, we can obtain only finitely many values of the fibre genus $g$ for a fixed $b$, because there are only finitely many primes $p$ dividing $b+1$. 

\begin{theorem} \label{thm:directly-Kodaira}
Assume that $p$ $(p \geq 2)$ divides $b+1$ and let $f \colon S^{\circ}_{b, \, p} \to \Sigma_b \times \Sigma_b$ be the degenerate Heisenberg cover corresponding to the group epimorphism $\varphi \colon \mathsf{P}_2(\Sigma_b) \to \mathsf{H}_{2b+1}(\mathbb{Z}_p)$ described above. Then, composing $f$  with the projections onto the two factors, we obtain two Kodaira fibrations $f_i \colon S^{\circ}_{b, \, p} \to \Sigma_b$, with the same fibre genus $g$ satisfying
\begin{equation} \label{eq:fibre-genus degenerate}
2g-2 = p^{2b+1}(2b-2+ \mathfrak{p}).
\end{equation}
Moreover, the invariants of $S^{\circ}_{b, \, p}$ are
\begin{equation} \label{eq:invariants-S-degenerate}
\begin{split}
c_1^2(S^{\circ}_{b, \, p}) & =  p^{2b+1}(2b-2) (4b-4 + 4 \mathfrak{p} - \mathfrak{p}^2) \\
c_2(S^{\circ}_{b, \, p}) & = p^{2b+1}(2b-2)(2b-2 + \mathfrak{p}),
\end{split}
\end{equation}
and so the slope and the signature of $S^{\circ}_{b, \, p}$ can be expressed as
\begin{equation} \label{eq:slope-signature-S-degenerate}
\begin{split}
\nu(S^{\circ}_{b, \, p}) & = \frac{c_1^2(S^{\circ}_{b, \, p})}{c_2(S^{\circ}_{b, \, p})} = 2+ \frac{2 \mathfrak{p} - \mathfrak{p}^2}{2b-2 + \mathfrak{p} } \\
\sigma(S^{\circ}_{b, \, p}) & = \frac{1}{3}\left(c_1^2(S^{\circ}_{b, \, p}) - 2 c_2(S^{\circ}_{b, \, p}) \right) = \frac{1}{3} p^{2b+1}(2b-2)(2 \mathfrak{p} - \mathfrak{p}^2).
\end{split}
\end{equation}
\end{theorem}
\begin{proof}
By construction, both group homomorphisms
\begin{equation}
\varphi_1 \colon \pi_1(\Sigma_b - \{p_1\}, \, p_2) \to \mathsf{Heis}(W, \, \omega), \quad \varphi_2 \colon \pi_1(\Sigma_b - \{p_2\}, \, p_1) \to \mathsf{Heis}(W, \, \omega)
\end{equation}
are surjective (in the case $p$ odd, this reflects the fact that the quotient map $\mathsf{Heis}(V, \, \omega) \to \mathsf{Heis}(W, \, \omega)$ identifies $\mathsf{r}_{1j}$ with $\mathsf{r}_{2j}$ and $\mathsf{t}_{1j}$ with $\mathsf{t}_{2j}$). This implies that the fibres of both fibrations $\pi_1 \circ f \colon S^{\circ}_{b, \, p} \to \Sigma_b$ and $\pi_2 \circ f \colon S^{\circ}_{b, \, p} \to \Sigma_b$ are connected, hence $f \colon S^{\circ}_{b, \, p} \to \Sigma_b \times \Sigma_b$ is a double Kodaira fibration. The fibre genus $g$ and the invariants of $S^{\circ}_{b, \, p}$ can be now computed by using Theorem \ref{thm:double-Kodaira-from-G} and Proposition \ref{prop:invariant-S-G}, setting $n=p$, $m_1=m_2=1$ and $|G| = |\mathsf{Heis}(W, \, \omega)| = p^{2b+1}$.
\end{proof}

In this situation, the bounds for the slope and the signature are provided by

\begin{proposition} \label{prop:invariants-degenerate-case}
Let $f \colon S^{\circ}_{b, \, p} \to \Sigma_b \times \Sigma_b$ be a double Kodaira fibration given by a degenerate Heisenberg cover as in Theorem \emph{\ref{thm:directly-Kodaira}}.
\begin{itemize}
\item[$\boldsymbol{(1)}$] We have
\begin{equation} \label{eq:slope-2}
2 < \nu(S^{\circ}_{b, \, p} ) \leq 2 + \frac{1}{3},
\end{equation}
with equality on the right holding precisely for $(b, \, p)=(2, \, 3)$.
\item[$\boldsymbol{(2)}$] The signature $\sigma(S^{\circ}_{b, \, p})$ is divisible by $16$ and we have
\begin{equation}
\sigma(S^{\circ}_{b, \, p}) \geq \sigma(S^{\circ}_{3, \,2})=128.
\end{equation}
\end{itemize}
\end{proposition}
\begin{proof}
The proof of $\boldsymbol{(2)}$ is immediate if we rewrite the second equation in \eqref{eq:slope-signature-S-degenerate} as
\begin{equation} \label{eq:rewrite-signature}
\sigma(S_{b, \, p}^{\circ}) = \frac{1}{3}(2b-2)p^{2b-1}(p^2-1),
\end{equation}
so we must only show $\boldsymbol{(1)}$. Assume first $p \geq 3$. Since by assumption $p$ divides $b+1$, we can write $b=kp-1$ for some positive integer $k\geq 1$ and, by substituting in \eqref{eq:slope-signature-S-degenerate}, we get
\begin{equation} \label{eq:nu-degenerate}
2 < \nu(S^{\circ}_{b, \, p}) = 2+\frac{p^2-1}{2kp^3-3p^2-p} \leq 2+\frac{p^2-1}{2p^3-3p^2-p},
\end{equation}
with equality on the right holding if and only if $k=1$. By computing its derivative, it is easily seen that the real function
\begin{equation}
 u(x)=\frac{x^2-1}{2x^3-3x^2-x}
\end{equation}
is strictly decreasing for $x \geq 3$; by \eqref{eq:nu-degenerate}, this implies that the maximum of $\nu(S^{\circ}_{b, \, p})$, for $p \geq 3$, is attained precisely when
$k=1$ and $p=3$, that corresponds to the pair $(b, \, p)=(2, \, 3)$ and yields for the slope the value  $2+1/3$. On the other hand, elementary calculations show that if $p=2$ the maximum is attained at $b=3$, that yields for the slope the value $2 + 1/6$. This completes the proof.  As an aside we note that, again by  \eqref{eq:nu-degenerate}, taking $p$ large enough we can make $\nu(S^{\circ}_{p-1, \, p})$ arbitrarily close to $2$.
\end{proof}

\begin{remark} \label{rmk:minimal-signature-decreases}
Degenerate Heisenberg covers  allow us to obtain a much smaller minimal signature than the construction with the symplectic form, namely $128$ instead of $2^4 \cdot 5^7$. This is due to the fact that, using the degenerate construction, the lowest admissible order of our Heisenberg group decreases from $5^9$ to $2^7$.

The price to pay is that now not all the pairs $(b, \, p)$ are suitable for the construction, but only those with $p$ dividing $b+1$, so we cannot get slope higher than $2+1/3$ in this way.
\end{remark}

\begin{remark} \label{rmk:minimum-base-genus}
The double Kodaira fibration $S^{\circ}_{3, \, 2} \to \Sigma_3 \times \Sigma_3$, besides minimizing the signature, also realizes the minimum fibre genus for the Kodaira fibrations constructed in this paper, namely $g = 289$.
\end{remark}

Let $\boldsymbol{\upomega} \colon \mathbb{N} \to \mathbb{N}$ be the arithmetic function counting the number of distinct prime factors of a positive integer, see \cite[p.335]{HarWr08}; therefore, as a consequence of Theorem \ref{thm:directly-Kodaira}, we obtain

\begin{corollary} \label{cor:number-kodaira-any-curve}
Let $\Sigma_b$ be any smooth curve of genus $b$. Then there exist at least one and at most finitely many  double Kodaira fibrations obtained as degenerate Heisenberg covers $f \colon S_{b, \, p}^{\circ} \to \Sigma_b \times \Sigma_b$. In fact, denoting by $\kappa(b)$ the number of such fibrations, we have
\begin{equation*}
\kappa(b) =\boldsymbol{\upomega}(b+1).
\end{equation*}
In particular,
\begin{equation*}
\limsup_{b \rightarrow + \infty} \kappa(b) = + \infty.
\end{equation*}
\end{corollary}
\begin{proof}
In view of Theorem \ref{thm:directly-Kodaira}, we only have to show that two different primes  dividing $b+1$ give rise to different Kodaira fibrations. Actually, something much stronger is true: the corresponding Kodaira fibred surfaces  are non-homeomorphic, because they have different signature. In fact, looking at \eqref{eq:rewrite-signature}, we see that, for fixed $b$, the signature $\sigma(S_{b, \, p}^{\circ})$ is strictly increasing in $p$.
\end{proof}

In particular, if $b=2$ the only possibility is $p=3$. This case provides (to our knowledge) the first ``double solution'' to a problem, posed by G. Mess, from Kirby's problem list in low-dimensional topology (\cite[Problem 2.18 A]{Kir97}), asking what is the smallest number $b$ for which there exists a real surface bundle over a surface with base genus $b$ and non-zero signature.

\begin{proposition} \label{prop:Kirby}
There exists a $4$-manifold $X$ $($namely, the real $4$-manifold underlying the complex surface $S^{\circ}_{2, \, 3})$ of signature $144$ that can be realized as a real surface bundle over a surface of genus $2$, with fibre  genus $325$, in two different ways.
\end{proposition}
This naturally leads to the following problem, that remains at the moment unsolved. 
\begin{question} \label{q:minimal-values}
What are the minimal possible fibre genus $f_{\mathrm{min}}$ and the minimum possible signature $\sigma_{\mathrm{min}}$ for a double Kodaira fibration $S \to \Sigma_2 \times \Sigma_2?$
\end{question}
Proposition \ref{prop:Kirby} implies $f_{\mathrm{min}} \leq 325$ and $\sigma_{\mathrm{min}} \leq 144$, but we do not know whether these inequalities are actually equalities.

\begin{remark} \label{rmk:BD02}
The first example of a real surface bundle over a surface of genus $2$ with non-zero signature was constructed in \cite{BD02}; it is a double Kodaira fibration with  $\sigma=16$ and 
\begin{equation}
(b_1, \, g_1)=(2, \, 25), \quad (b_2, \, g_2) = (9,\, 4).
\end{equation}
 \end{remark}
\medskip

Finally, let us say something about the deformations of the double Kodaira fibrations
$f \colon S^{\circ}_{b, \, p} \to \Sigma_b \times \Sigma_b$. Any such fibration is very simple and branched over the diagonal, which is the graph of the identity $\Sigma_b \to \Sigma_b$. If we denote by $\mathfrak{M}^{\circ}_{b, \, p}$ the connected component of the Gieseker moduli space of surfaces of general type containing (the class of) $S^{\circ}_{b, \, p}$, and by $\mathcal{M}_b$ the moduli space of smooth curves of genus $b$, applying \cite[Theorem 1.7]{Rol10} we obtain the following
\begin{proposition} \label{prop:moduli}
There is a natural map
\begin{equation}
\mathcal{M}_b \to \mathfrak{M}^{\circ}_{b, \, p}
\end{equation}
that is an isomorphism on closed points.
\end{proposition}
Morally, since the branch divisor $\Delta$ of $f \colon S^{\circ}_{b, \, p} \to \Sigma_b \times \Sigma_b$ is rigid, all deformations of $S^{\circ}_{b, \, p}$ are realized by deformations of $\Sigma_b \times \Sigma_b$ preserving the diagonal, hence by deformations of $\Sigma_b$. In fact, it is well-known that the natural morphism of deformation functors
\begin{equation*}
\mathrm{Def}_{\Sigma_b} \times \mathrm{Def}_{\Sigma_b} \to \mathrm{Def}_{\Sigma_b \times \Sigma_b}
\end{equation*}
is an isomorphism, see \cite[p. 74]{Sern06}.

\section*{Acknowledgments}
Both authors warmly thank F. Catanese and G. Pirola for several stimulating conversations on the topic of this work, and the anonymous referee for  constructive recommendations and remarks. F. Polizzi was partially supported by GNSAGA-INdAM. He il grateful to A. Rapagnetta for his help on the proof of Proposition \ref{prop:integral-cohomology-conf}, to the Mathematics Stack Exchange user \emph{awllower} for his help on the proof of Proposition \ref{prop:heis-forma-exist}, to J. van Bon for helpful discussions on extra-special groups, and to D. Holt and W. Sawin for their precious comments and suggestions in the MathOverflow thread \\
\verb|https://mathoverflow.net/questions/298174|.


\begin{thebibliography} {999999}

\bibitem[AAG09]{AAG09}
F. Altunbulak Aksu - D. J. Green: Essential cohomology for elementary abelian $p$-groups, \emph{J. Pure Appl. Algebra} $\boldsymbol{213}$ (2009), 2238--2243. 

\bibitem[At69]{At69}
M. F. Atiyah: The signature of fibre bundles, in \emph{Global Analysis $($Papers in honor of K. Kodaira$)$} 73--84, Univ. Tokyo Press (1969).

\bibitem[Az15]{Az15}
H. Azam: Cohomology groups of configuration spaces of Riemann surfaces, \emph{Bull. Math. Soc. Sci. Math. Roumanie} $\boldsymbol{58}$(106) (2015), no. 1, 33--47.

\bibitem[BHPV03]{BHPV03}
W. Barth, K. Hulek, C.A.M. Peters, A. Van de Ven, \textit{Compact
Complex Surfaces}. Grundlehren der Mathematischen Wissenschaften,
Vol $\boldsymbol{4}$, Second enlarged edition, Springer-Verlag, Berlin,
2003.

\bibitem[Be82]{Be82}
A. Beauville: L'inegalit{\'e} $p_g \geq 2q-4$ pour les surfaces de
type g{\'e}n{\'e}rale, \emph{Bull. Soc. Math. de France}
$\boldsymbol{110}$ (1982), 343-346.

\bibitem[Be96]{Be96}
A. Beauville: \emph{Complex algebraic surfaces}, Cambridge
University Press 1996.

\bibitem[Bel04]{Bel04}
P. Bellingeri: On presentations of surface braid groups, \emph{J. Algebra} $\boldsymbol{274}$ (2004), 543--563.

\bibitem[BC92]{BC92}
D. J. Benson, J. F. Carlson: The cohomology of extraspecial groups, \emph{Bull. London Math. Soc.} $\boldsymbol{24}$ (1992), 209--235.


\bibitem[Bir69]{Bir69}
J. Birman: On braid groups, \emph{ Comm. Pure Appl. Math.} $\boldsymbol{22}$ (1969), 41--72.

\bibitem[Bre93]{Bre93}
G. E. Bredon: \emph{Topology and Geometry}, Graduate Texts in Mathematics $\boldsymbol{139}$, Springer 1993.

\bibitem[Breg18]{Breg18}
C. Bregman: On Kodaira fibrations with invariant cohomology, e-print
$\mathsf{arXiv1811.00584v1}$ (2018).

\bibitem[BD02]{BD02}
J. Bryan, R. Donagi: Surface bundles over surfaces of small genus, \emph{Geom. Topol.} $\boldsymbol{6}$ (2002), 59--67.

\bibitem[BDS01]{BDS01}
J. Bryan, R. Donagi, A. I. Stipsicz: Surface bundles: some interesting examples, \emph{Turkish J. Math.} $\boldsymbol{25}$ (2001), no. 1, 61--68.

\bibitem[CatRol09]{CatRol09}
F. Catanese, S. Rollenske: Double Kodaira fibrations, \emph{J. Reine Angew. Math.} $\boldsymbol{628}$ (2009), 205--233.

\bibitem[Cat17]{Cat17}
F. Catanese: Kodaira fibrations and beyond: methods for moduli theory, \emph{Japan. J. Math.} $\boldsymbol{12}$ (2017), no. 2, 91--174.

\bibitem[CHS57]{CHS57}
S. S. Chern, F. Hirzebruch, J. P. Serre: On the index of a fibred manifold, \emph{Proc. Amer. Math. Soc.} $\boldsymbol{8}$ (1957), 587--596

\bibitem[DHW12]{DHW12}
K. Dekimpe, M. Hartl, S. Wauters: A seven terms exact sequence for the cohomology of a group extension, \emph{J. Algebra} $\boldsymbol{369}$ (2012), 70-95.

\bibitem[DM69]{DM69}
P. Deligne, D. Mumford: The irreducibility of the space of curves of given genus, \emph{Publ. Math. IHES} $\boldsymbol{36}$ (1969), 75--109

\bibitem[Eh51]{Eh51}
C. Ehresmann: Les connexions infinit{\'e}simales dans un espace fibr{\'e} diff{\'e}rentiable, \emph{Colloque de topologie $($espaces fibr{\'e}s$)$, Bruxelles, 1950}, 29--55. Georges Thone, Li{\`e}ge; Masson et Cie., Paris (1951).

\bibitem[En98]{En98}
H. Endo: A construction of surface bundles over surfaces with non-zero signature, \emph{Osaka J. Math.} $\boldsymbol{35}$ (1998), 915--930.

\bibitem[EKKOS02]{EKKOS02}
H. Endo, M. Korkmaz, D, Kotschick, B. Ozbagci, A. Stipsicz: Commutators, Lefschetz fibrations and the signature of surface bundles, \emph{Topology} $\boldsymbol{41}$ no. 5 (2002), 961--977.

\bibitem[FN62]{FN62}
E. Fadell, L. Neuwirth: Configuration spaces, \emph{Math. Scand.} $\boldsymbol{10}$ (1962), 111--118.

\bibitem[FiGr65]{FiGr65}
W. Fischer, H. Grauert: Lokal-triviale Familien kompakter komplexer Mannigfaltigkeiten,
\emph{Nachr. Akad. Wiss. G\"{o}ttingen Math.-Phys. Kl. II} (1965), 89--94.

\bibitem[Fl17]{Fl17}
L. Flaplan: Monodromy of Kodaira fibrations of genus $3$, e-print
$\mathsf{arXiv1709.03164v1}$ (2017).

\bibitem[GG04]{GG04}
D. L. Gon\c{c}alves, J. Guaschi: On the structure of surface pure braid groups, \emph{J. Pure Appl. Algebra} $\boldsymbol{186}$ (2004), 187--218.

\bibitem[GDH91a]{GDH91a}
G. Gonz{\'a}lez-D{\'i}ez, W. Harvey: On complete curves in moduli spaces I, \emph{Math. Proc. Camb. Phil. Soc.} $\boldsymbol{110}$ (1991), 461--466.

\bibitem[GDH91b]{GDH91b}
G. Gonz{\'a}lez-D{\'i}ez, W. Harvey: On complete curves in moduli spaces II, \emph{Math. Proc. Camb. Phil. Soc.} $\boldsymbol{110}$ (1991), 467--472.

\bibitem[Gor07]{Gor07}
D. Gorenstein: \emph{Finite Groups}, reprinted edition by the AMS Chelsea Publishing, 2007.

\bibitem[HarWr08]{HarWr08}
G. H. Hardy, E. M. Wright: \emph{An introduction to the theory of numbers}, sixth edition, Oxford University Press 2008.

\bibitem[Hat02]{Hat02}
A. Hatcher: \emph{Algebraic topology}, Cambridge University Press 2002.

\bibitem[Hil00]{Hil00}
J. A. Hillman: Complex surfaces which are fibre bundles, \emph{Topology Appl.} $\boldsymbol{100}$ (2000), no. 2--3, 187--191.

\bibitem[Hir69]{Hir69}
F. Hirzebruch: The signature of ramified covers, in \emph{Global Analysis $($Papers in honor of K. Kodaira$)$} 253--265, Univ. Tokyo Press (1969).

\bibitem[Is08]{Is08}
M. Isaacs: \emph{Finite groups theory}, Graduate Studies in Mathematics $\boldsymbol{92}$, American Mathematical Society 2008.

\bibitem[Kas68]{Kas68}
A. Kas: On deformations of a certain type of irregular algebraic surface, \emph{Amer. J. Math.} $\boldsymbol{90}$ (1968), 789--804.

\bibitem[Kir97]{Kir97}
R. Kirby (editor): \emph{Problems in low-dimensional topology}, AMS/IP Stud. Adv. Math., 2.2, Geometric topology (Athens, GA, 1993), 35--473, Amer. Math. Soc., Providence, RI, 1997.

\bibitem[Kod67]{Kod67}
K. Kodaira: A certain type of irregular, algebraic surfaces, \emph{J. Anal. Math.} $\boldsymbol{19}$ (1967), 207--215.

\bibitem[Kr94]{Kr94}
I. Kriz: On the rational homotopy type of configuration spaces,  \emph{Ann. of Math.} (2) $\boldsymbol{139}$ (1994), no. 2, 227--237.

\bibitem[LeBrun00]{LeBrun00}
C. LeBrun: Diffeomorphisms, symplectic forms and Kodaira fibrations, \emph{Geom. Topol.} $\boldsymbol{4}$ (2000), 451--456.

\bibitem[L17]{L17}
J. A Lee: Surface bundles over surfaces with a fixed signature, \emph{J. Korean Math. Soc.} $\boldsymbol{54}$ (2017), no. 2, 545--561.

\bibitem[LLR17]{LLR17}
J. A Lee, M. L\"{o}nne, S. Rollenske: Double Kodaira fibrations with small signature, e-print
$\mathsf{arXiv1711.01792v1}$ (2017).

\bibitem[Liu96]{Liu96}
K. Liu: Geometric height inequalities, \emph{Math. Res. Lett.} $\boldsymbol{3}$ (1996), no. 5, 693--702.

\bibitem[ML95]{ML95}
S. Mac Lane: \emph{Homology}, Classics in Mathematics, Springer 1995.

\bibitem[Mey73]{Mey73}
W. Meyer: Die Signatur von Fl\"{a}chenb\"{u}ndeln, \emph{Math. Ann.} $\boldsymbol{201}$ (1973), 239--264. 

\bibitem[MilSt74]{MilSt74}
J. W. Milnor, J. D. Stasheff: \emph{Characteristic classes}, Annals of Mathematical Studies $\boldsymbol{76}$, Princeton University Press 1974.

\bibitem[Mo96]{Mo96}
J. D. Moore: \emph{Lectures on Seiberg-Witten invariants}, Lecture Notes in Mathematics $\boldsymbol{1629}$, Springer 1996.

\bibitem[Pig19]{Pig19}
R. Pignatelli: Quotient of the square of a curve by a mixed action, further quotients and Albenese morphisms, \emph{Rev. Mat. Complut.} (2019), doi:10.1007/s13163-019-00337-8.

\bibitem[Pol18]{Pol18}
F. Polizzi: Monodromy representations and surfaces with maximal Albanese dimension, \emph{Boll. Unione Mat. Ital.} $\boldsymbol{11}$ (2018), no. 1, 107--119.

\bibitem[Rol10]{Rol10}
S. Rollenske: Compact moduli for certain Kodaira fibrations, \emph{Ann. Scuola Norm. Sup. Pisa Cl. Sci.} (5) Vol. IX (2010), 851-874.

\bibitem[Rot02]{Rot02}
J. J. Rotman: \emph{Advanced modern algebra}, 2nd edition, Graduate Studies in Mathematics $\boldsymbol{114}$, American Mathematical Society 2002.

\bibitem[Sern06]{Sern06}
E. Sernesi, \textit{Deformations of Algebraic Schemes}. Grundlehren
der Mathematischen Wissenschaften, Vol. $\boldsymbol{334}$,
Springer-Verlag, Berlin, 2006.

\bibitem[S70]{S70}
G. P. Scott: Braid groups and the groups of homeomorphisms of a surface, \emph{Proc. Camb. Phil. Soc.} $\boldsymbol{68}$ (1970), 605--617.

\bibitem[St02]{St02}
A. I. Stipsicz: Surface bundles with nonvanishing signature, \emph{Acta Math. Hungar.} $\boldsymbol{95}$ (2002), no. 4, 299--307.

\bibitem[To93]{To93}
B. Totaro: Configuration spaces of algebraic varieties, \emph{Topology} $\boldsymbol{35}$ (1993), no. 4, 1057--1067.

\bibitem[Tr16]{Tr16}
P. Tretkoff: \textit{Complex ball quotients and line arrangements in the projective plane}, Mathematical Notes $\boldsymbol{51}$, Princeton University Press 2016.

\bibitem[We14]{We14}
S. H. Weintraub: \textit{Fundamentals of algebraic topology}, Graduate Texts in Mathematics $\boldsymbol{270}$, Springer 2014.

\bibitem[Zaal95]{Zaal95}
C. Zaal: Explicit complete curves in the moduli space of curves of genus three, \emph{Geom. Dedicata}
$\boldsymbol{56}$ (1995), no. 2, 185--196.


\end{thebibliography}
\end{document}